\documentclass[12pt]{article}
\usepackage{ae,aecompl}

\usepackage[T1]{fontenc}
\usepackage[a4paper]{geometry}
\usepackage[active]{srcltx}
\usepackage{array}
\usepackage{float}
\usepackage{subfigure}
\usepackage[colorlinks=true,linkcolor=blue,citecolor=blue,urlcolor=blue]{hyperref}
\usepackage[figuresright]{rotating}
\usepackage{calc}
\usepackage{textcomp}
\usepackage{multirow}
\usepackage{algorithm2e}
\usepackage{amsmath}
\usepackage{amsthm}
\usepackage{amssymb}
\usepackage[capitalise]{cleveref}
\usepackage{mathptmx}
\usepackage{graphicx}
\usepackage{esint}
\usepackage{todonotes}
\usepackage{breakurl}
\usepackage{algpseudocode}
\usepackage{url}
\usepackage{enumitem}
\usepackage{rotating}

\newcommand{\R}{\mathbb{R}}
\newcommand{\RR}{\mathbb{R}\cup\{+\infty\}}
\newcommand{\dom}{\operatorname{dom}}

\newtheorem{theorem}{Theorem}[section]
\newtheorem{fact}{Fact}[section]
\theoremstyle{definition}
\newtheorem{definition}{Definition}[section]
\newtheorem{lemma}{Lemma}[section]
\newtheorem{proposition}{Proposition}[section]

\newtheorem{experiment}{Experiment}
\theoremstyle{remark}
\newtheorem{remark}{Remark}[section]
\newtheorem{example}{Example}[section]

\newcommand{\qede}{\hspace*{\fill}$\Diamond$\medskip}

\makeatletter

\renewcommand{\labelenumi}{(\roman{enumi})}

\makeatother

\begin{document}
\title{The Boosted DC Algorithm \\for nonsmooth functions }

\author{Francisco J. Arag\'on Artacho\footnote{Department of Mathematics, University of Alicante, Alicante, Spain; email: francisco.aragon@ua.es},\,
 Phan T. Vuong\footnote{Faculty of Mathematics, 	University of Vienna,  	Oskar-Morgenstern-Platz 1, 1090 Vienna,
 	Austria; email: vuong.phan@univie.ac.at}}

\date{\today}

\maketitle
\begin{abstract}
The Boosted Difference of Convex functions Algorithm (BDCA) was recently proposed for minimizing smooth difference of convex (DC) functions. BDCA accelerates the convergence of the classical Difference of Convex functions Algorithm (DCA) thanks to an additional line search step. The purpose of this paper is twofold. Firstly, to show that this scheme can be generalized and successfully applied to certain types of nonsmooth DC functions, namely, those that can be expressed as the difference of a smooth function and a possibly nonsmooth one. Secondly, to show that there is complete freedom in the choice of the trial step size for the line search, which is something that can further improve its performance. We prove that any limit point of the BDCA iterative sequence is a critical point of the  problem under consideration, and that the corresponding  objective value is monotonically decreasing and convergent. The global convergence and convergent rate of the iterations are obtained under the Kurdyka--\L{}ojasiewicz property. Applications and numerical experiments for two problems in data science are presented, demonstrating that BDCA outperforms DCA. Specifically, for the Minimum Sum-of-Squares Clustering problem, BDCA was on average sixteen times faster than DCA, and for the Multidimensional Scaling problem, BDCA was three times faster than DCA.
\end{abstract}

\paragraph{\bf Keywords:} Difference of convex functions;  Boosted Difference of Convex functions Algorithm; Kurdyka--\L{}ojasiewicz property;
Clustering problem; Multidimensional Scaling problem.

\paragraph{\bf AMS subject classifications:} 65K05, 65K10, 90C26, 47N10

\section{Introduction}

In this paper, we are interested in the following DC (difference of convex) optimization problem
\begin{equation}
\left(\mathcal{P}\right)\ \text{\ensuremath{\underset{x\in\mathbb{R}^{m}}{\textrm{minimize}}}}\; g(x)-h(x)=: \phi(x),\label{eq:DC_2norm}
\end{equation}
where $g:\R^m\to\RR$ and $h:\R^m\to\RR$ are proper convex functions, with the conventions
\begin{gather*}
(+\infty)-(+\infty)=+\infty,\\
(+\infty)-\lambda=+\infty\quad\text{and}\quad\lambda-(+\infty)=-\infty,\quad\forall\lambda\in{]-\infty,+\infty[}.
\end{gather*}

For solving~$\left(\mathcal{P}\right)$, one usually applies the well-known DC Algorithm (DCA) \cite{An2014,an_numerical_1996,TT97} (see \cref{sec:DCA}). DC programming and the DCA have been investigated and developed for more than 30 years \cite{An2018}.
The DCA has been successfully applied in many fields, such as
machine learning, financial optimization, supply chain management and telecommunication \cite{tao2005dc,an_numerical_1996,An2018}.
If both functions $g$ and $h$ are differentiable, then the Boosted DC Algorithm (BDCA) developed in~\cite{BDCA2018} can be applied to accelerate the convergence of DCA.
Numerical experiments with various biological data sets in \cite{BDCA2018} showed that BDCA outperforms DCA,
being on average more than four times faster in both computational time and the number of iterations.
This advantage has been also confirmed when applying BDCA to the Indefinite Kernel Support Vector Machine problem \cite{XuXue2017}.

The purpose of the present paper is to develop a version of BDCA when the function~$\phi$ is not differentiable. Unfortunatelly, when $g$ is not differentiable,
the direction used by BDCA may no longer be a descent direction (see \cref{ex:failure}). For this reason, we shall restrict ourselves to the case where
$g$ is assumed to be differentiable but $h$ is not.
The motivation for this study comes from
many applications of DC programming where the objective function
is the difference of a smooth convex function and a nonsmooth convex function.
We mention here the Minimum
Sum-of-Squares Clustering problem~\cite{CYY2018},
the Bilevel Hierarchical Clustering problem \cite{Nam2018},  the Multicast Network Design problem \cite{Geremew2018}, and the Multidimensional Scaling problem \cite{Tao01}, among others.

The paper is organized as follows. In \cref{sec:Preliminaries}, we  recall some basic concepts and properties of convex analysis.
As we are working with nonconvex and nonsmooth functions, we need some tools
from variational analysis for generalized differentiability. 

Our main contributions are in \cref{sec:DCA}, where
we propose a nonsmooth version of the BDCA introduced in \cite{BDCA2018}.
More precisely, we prove that the point generated by the DCA
provides a descent direction for the objective function
at this point, even at points where the function $h$ is not differentiable. This is the key property allowing us to employ a simple line search along the descent direction, which permits to achieve a larger decrease
in the value of the objective function.

In \cref{sec:KL}, we investigate the global convergence and convergence rate of the BDCA.
The convergence analysis relies on the Kurdyka--\L{}ojasiewicz inequality. These concepts of real algebraic
geometry were introduced by  \L{}ojasiewicz \cite{lojasiewicz1965ensembles} and Kurdyka \cite{Kurdyka}
 and later developed in the nonsmooth setting by Bolte, Daniilidis, Lewis and Shiota \cite{bolteArisLewis2007},
 and Attouch, Bolte, Redont, and Soubeyran \cite{Attouch2010}, among many others \cite{AnNam2017,attouch2009convergence,BB2018,bolte2007lojasiewicz,Bolte2013,Noll2014}.

In \cref{sec:numa}, we begin by introducing a self-adaptive strategy for choosing the trial step size for the line search step. We show that this strategy permits to further improve the numerical results obtained in~\cite{BDCA2018} for the above-mentioned problem arising in biochemistry, being BDCA almost seven times faster than DCA on average. Next, we present an application of BDCA to two important classes of DC programming problems in engineering: the Minimum Sum-of-Squares Clustering problem and the Multidimensional Scaling problem. We present some numerical experiments on large data sets, both with real and randomly generated data, which clearly show that BDCA outperforms DCA. Namely, on average, BDCA was sixteen times faster than DCA for the Minimum Sum-of-Squares Clustering and three times faster for the Multidimensional Scaling problems.
We conclude the paper with some remarks and future research directions in the last section.

\section{Preliminaries} \label{sec:Preliminaries}
Throughout this paper, the inner product of two vectors $x,y\in\mathbb{R}^{m}$
is denoted by $\langle x,y\rangle$, while $\|\cdot\|$ denotes the
induced norm, defined by $\|x\|=\sqrt{\langle x,x\rangle}$. The closed ball of center $x$ and
radius $r>0$ is denoted by $\mathbb{B}(x,r)$.
\subsection{Tools of convex and variational analysis}
In this subsection, we recall some basic concepts and results of
convex analysis and generalized differentiation for nonsmooth
functions, which will be used in the sequel.

For an extended real-valued function $f:\mathbb{R}^{m}\to\mathbb{R} \cup \{+\infty\}$, the domain of $f$ is the
set
$$
\dom f = \left\lbrace x \in \mathbb{R}^{m}:  f(x) < +\infty  \right\rbrace.
$$
The function $f$ is said to be proper if its domain is nonempty.
It is said to be \emph{convex} if
$$
f(\lambda x+(1-\lambda)y)\leq\lambda f(x)+(1-\lambda)f(y)\quad\text{for all }x,y\in\mathbb{R}^{m}\text{ and }\lambda\in{]0,1[},
$$
and $f$ is said to be \emph{concave} if $-f$ is convex.
Further, $f$ is called \emph{strongly convex} with modulus $\rho>0$
if for all $x,y\in\mathbb{R}^{m}$ and $\lambda\in{]0,1[}$,
\[
f(\lambda x+(1-\lambda)y)\leq\lambda f(x)+(1-\lambda)f(y)-\frac{1}{2}\rho\lambda(1-\lambda)\|x-y\|^{2},
\]
or, equivalently, when $f-\frac{\rho}{2}\|\cdot\|^{2}$ is convex.
The function $f$ is said to be \emph{coercive} if ~$f(x)\to+\infty$
whenever $\left\Vert x\right\Vert \to+\infty.$ The gradient of a function $f:\mathbb{R}^{m}\to\RR$ which is differentiable
at some point $x$ in the interior of $\dom f$ is denoted by $\nabla f(x)$. We denote by $f'(x,d)$ the one-sided directional derivative of $f$ at $x\in\dom f$ for the direction $d\in\R^m$, defined as
 $$
 f'(x;d) := \lim_{t \downarrow 0} \frac{f(x+ t d)-f(x)}{t}.
 $$

A function $F:\mathbb{R}^{m}\to\mathbb{R}^{m}$
is said to be \emph{monotone} when
\[
\langle F(x)-F(y),x-y\rangle\geq0\quad\text{for all }x,y\in\mathbb{R}^{m}.
\]
Further, $F$ is called \emph{strongly monotone} with modulus $\rho>0$
when
\[
\langle F(x)-F(y),x-y\rangle\geq\rho\|x-y\|^{2}\quad\text{for all }x,y\in\mathbb{R}^{m}.
\]
 The function $F$ is called \emph{Lipschitz continuous} if there
is some constant $L\geq0$ such that
\[
\|F(x)-F(y)\|\leq L\|x-y\|,\quad\text{for all }x,y\in\mathbb{R}^{m},
\]
and $F$ is said to be locally Lipschitz continuous if, for every $x$ in
$\mathbb{R}^{m}$, there exists a neighborhood $U$ of $x$ such
that $F$ restricted to $U$ is Lipschitz continuous.

We have the following well-known result (see, e.g., \cite[Exercise~12.59]{RW}).
\begin{fact}\label{fact:strongly_convex}
A function $f:\mathbb{R}^{m}\to\mathbb{R}\cup \{+\infty\}$ is strongly convex with modulus $\rho$ if and only if $\partial f$ is strongly monotone with modulus $\rho$.
\end{fact}

The \emph{convex subdifferential} $\partial f (\bar x)$  of a function $f$ at $\bar x \in \mathbb{R}^m$ is defined at any point $\bar x \in \dom f$ by
$$
\partial f(\bar x)=\left\lbrace u \in \mathbb{R}^m \mid f(x)-f(\bar x) \geq \left\langle u,x-\bar x\right\rangle, \forall x \in \mathbb{R}^m  \right\rbrace,
$$
and is empty otherwise.

When dealing with nonconvex and nonsmooth functions, we have to consider subdifferentials
more general than the convex one.
%
%
%
One of the most widely used constructions is the Clarke subdifferential, which can be defined in several (equivalent) ways (see, e.g.,~\cite{Clarke}).
For a given locally Lipschitz continuous function $f:\mathbb{R}^{m}\to\mathbb{R} \cup \{+\infty\}$,
the \emph{Clarke subdifferential} of $f$ at $\bar{x}$   is given by
$$
\partial_C f(\bar{x}) =\text{co} \left\{\lim_{x\to\bar x,\, x\not\in \Omega_f}\nabla f(x)\right\},
$$
where $\rm{co}$ stands for the convex hull and $\Omega_f$ denotes the set of Lebesgue measure zero (by Rademacher's Theorem) where $f$ fails to be differentiable. When $f$ is also convex on a neighborhood of $\bar x$, then $\partial_C f(\bar{x})=\partial f(\bar{x})$ (see \cite[Proposition~2.2.7]{Clarke}).
%

Clarke subgradients are generalizations of the
usual gradient of smooth functions. Indeed, if $f$ is strictly differentiable at $x$, 
we have
$$
\partial_C f(x)=\left\lbrace \nabla f(x)\right\rbrace,
$$
see~\cite[Proposition~2.2.4]{Clarke}.
However, it should be noted that if $f$ is only Fr\'echet differentiable at $x$,
 then $\partial_C f(x)$ can contain points other than $\nabla f(x)$ (see, e.g.~\cite[Example~2.2.3]{Clarke}).

The next basic formulas facilitate the calculation of the Clarke subdifferential.


\begin{fact}[Basic calculus]\label{Clarkesumrule} The following assertions hold:
\begin{enumerate}
\item For any scalar $s$, one has
$$
\partial_C(sf)(x)= s\partial_C f(x).
$$
\item $\partial_C (f+g)(x)\subset\partial_C f(x)+\partial_C g(x)$, and equality holds if either $f$ or $g$ is strictly differentiable.
\end{enumerate}
\end{fact}
\begin{proof}
See~\cite[Propositions 2.3.1 and 2.3.3]{Clarke}. For the last assertion, see~\cite[Corollary~1, p.~39]{Clarke}.
\end{proof}

\subsection{Assumptions}
Throughout this paper, the following two assumptions are made.

\paragraph{Assumption 1} Both functions $g$ and $h$ are strongly convex with modulus $\rho >0 $.

\paragraph{Assumption 2} The function $h$ is subdifferentiable at every point in $\dom  h$; i.e., $\partial h(x)\neq\emptyset$ for all $x\in\dom  h$. The function $g$ is continuously differentiable on an open set containing $\dom h$  and
\begin{equation}
\inf_{x\in\mathbb{R}^{m}}\phi(x)>-\infty.\label{eq:phi_bounded_below}
\end{equation}


Under these assumptions, the next necessary optimality condition holds.

\begin{fact}[First-order necessary optimality condition]
If $x^* \in \dom \phi$ is an optimal solution of problem $\left(\mathcal{P}\right)$ in~\eqref{eq:DC_2norm}, then
\begin{equation} \label{stationaryPoint}
\partial h(x^*) = \left\lbrace \nabla g(x^*)
 \right\rbrace.
\end{equation}
\end{fact}
\begin{proof}
See {\cite[Theorem~3']{Toland79}}.
\end{proof}
\bigskip

Any point satisfying condition \eqref{stationaryPoint} is called a {\em stationary point} of $\left(\mathcal{P}\right)$.
One says that $ \bar{x}$ is a {\em critical point} of $\left(\mathcal{P}\right)$ if
$$
\nabla g(\bar{x}) \in \partial h(\bar{x}).
$$
It is obvious that every stationary point $x^*$
is a critical point, but the converse is not true in general.

\begin{example}\label{ex:function}
Consider the DC function $\phi:\mathbb{R}^m\to\mathbb{R}$ defined for $x\in\mathbb{R}^m$ by
$$\phi(x):=\|x\|^2+\sum_{i=1}^mx_i-\sum_{i=1}^m|x_i|.$$
It is not difficult to check that $\phi$ has $2^m$ critical points, namely, any $x\in\{-1,0\}^m$, and only one stationary point $x^*:=(-1,-1,\ldots,-1)$, which is the global minimum of $\phi$.\qede
\end{example}

\section{DCA and BDCA\label{sec:DCA}}


The key idea of the DCA to solve problem~$\left(\mathcal{P}\right)$ in~\eqref{eq:DC_2norm} is to
approximate the concave part $-h$ of the objective function~$\phi$
by its affine majorization, and then minimize the resulting convex
function. The algorithm proceeds as follows.\medskip
%
\begin{center}
\bgroup 	\renewcommand\theenumi{\arabic{enumi}.} 	 \renewcommand\labelenumi{\theenumi}
\fbox{%
\begin{minipage}[t]{.98\textwidth}%
\textbf{DCA} (\emph{DC Algorithm}) \cite{an_numerical_1996}
\begin{enumerate}
\item Let $x_{0}$ be any initial point and set $k:=0$.
\item Select $u_k \in \partial h(x_k)$ and solve the strongly convex optimization problem
\[
\left(\mathcal{P}_{k}\right)\ \underset{x\in\mathbb{R}^{m}}{\textrm{minimize}}\; g(x)-\langle u_k,x\rangle
\]
 to obtain its unique solution $y_{k}$.
\item If $y_{k}=x_{k}$ then STOP and RETURN $x_{k}$, otherwise set $x_{k+1}:=y_{k}$,
set $k:=k+1$, and go to Step~2.\medskip{}
 \end{enumerate}
\end{minipage}} \egroup{}
\end{center}\medskip


Let us introduce the algorithm we propose for solving problem $(\mathcal{P})$, which
we call \emph{BDCA} (\emph{Boosted DC Algorithm}).
The algorithm is a nonsmooth version of the one proposed in \cite{BDCA2018}, except for a small but relevant modification in Step 4, where now we give total freedom to the initial value for the backtracking line search used for finding an appropriate value of the step size $\lambda_k$. In \cref{sec:numa}, we demonstrate that this seemingly minor change permits \emph{smarter choices} of the initial value than simply using a constant value $\overline{\lambda}$. We have also replaced $\lambda_k$ in the right-hand side of the line search inequality by~$\lambda_k^2$, which allows us to remove the inconvenient assumption $\rho > \alpha$ (see \cite[Remark 3]{BDCA2018} for more details).\medskip
\begin{center}
\bgroup 	\renewcommand\theenumi{\arabic{enumi}.} 	 \renewcommand\labelenumi{\theenumi}
\fbox{%
\begin{minipage}[t]{.98\textwidth}%
\textbf{BDCA} (\emph{Boosted DC Algorithm})
\begin{enumerate}
\item Fix $\alpha>0$ and $0<\beta<1$. Let $x_{0}$
be any initial point and set $k:=0$.
\item Select $u_k \in \partial h(x_k)$ and solve the strongly convex optimization problem
\[
\left(\mathcal{P}_{k}\right)\ \underset{x\in\mathbb{R}^{m}}{\textrm{minimize}}\; g(x)-\langle u_k,x\rangle
\]
 to obtain its unique solution $y_{k}$.
\item Set $d_{k}:=y_{k}-x_{k}$. If $d_{k}=0$, STOP and RETURN $x_{k}$.
Otherwise, go to Step~4.
\item Choose any $\overline{\lambda}_k\geq 0$. Set $\lambda_k:=\overline{\lambda}_k$. \\WHILE $\phi(y_{k}+\lambda_{k}d_{k})>\phi(y_{k})-\alpha\lambda_{k}^2\|d_{k}\|^{2}$
DO $\lambda_{k}:=\beta\lambda_{k}$.
\item Set $x_{k+1}:=y_{k}+\lambda_{k}d_{k}$, set $k:=k+1$, and go to Step~2.\medskip{}
 \end{enumerate}
\end{minipage}}
\egroup{}
\end{center}\medskip

Observe that if one sets $\overline{\lambda}_{k}=0$, the iterations of the BDCA and the DCA coincide. Hence, our convergence results for the BDCA apply in particular to the DCA. In the following proposition we show that $d_{k}:=y_{k}-x_{k}$ is a descent direction for $\phi$ at $y_{k}$. Since the value of $\phi$ is always reduced at $y_k$ with respect to that at $x_k$, one can achieve a larger decrease by moving along the direction $d_k$. This simple fact, which is the key idea of the BDCA, improves the performance of the DCA in many applications (see \cref{sec:numa}).

\begin{proposition}\label{prop:main_inequality}
For all $k\in\mathbb{N}$, the following holds:
\begin{itemize}
\item [(i)] $\phi(y_{k})\leq\phi(x_{k})-\rho\|d_{k}\|^{2}$;
\item[(ii)] $\phi'(y_k;d_k) \leq-\rho\|d_{k}\|^{2}$;
\item[(iii)] there is some $\delta_k > 0$ such that
\begin{equation*}
\phi\left(y_{k}+\lambda d_{k}\right)\leq\phi(y_{k})- \alpha \lambda^2 \|d_{k}\|^{2},\quad\text{for all }\lambda\in[0,\delta_k],\label{eq:main_inequality_bdca}
\end{equation*}
so the backtracking Step 4 of BDCA terminates finitely.
\end{itemize}
\end{proposition}
\begin{proof}
The proof of~(i) is similar to the one of~\cite[Proposition~3]{BDCA2018} and is therefore omitted.
To prove~(ii), pick any $v\in\partial h(y_k)$. Note that the one-sided directional derivative $\phi'(y_k;d_k)$ is
given by
\begin{align}\label{directionalDerivative}
\phi'(y_k;d_k) \nonumber
&= \lim_{t \downarrow 0} \frac{\phi(y_k+ t d_k)-\phi(y_k)}{t}\\\nonumber
&= \lim_{t \downarrow 0} \frac{g(y_k+ t d_k)-g(y_k)}{t} - \lim_{t \downarrow 0} \frac{h(y_k+ t d_k)-h(y_k)}{t}\\
&\leq \left\langle \nabla g(y_k), d_k \right\rangle  - \left\langle v, d_k \right\rangle,
\end{align}
by convexity of $h$.
Since $y_{k}$ is the unique solution of the strongly convex
problem $\left(\mathcal{P}_{k}\right)$, we have
\[
\nabla g(y_{k})=u_k \in \partial h(x_k).
\]
The function $h$ is strongly convex with constant $\rho$.
This implies, by \cref{fact:strongly_convex}, that $\partial h$ is strongly monotone with constant $\rho$.
Therefore, since $v \in \partial h(y_k)$, it holds
\[
\langle u_k-v,x_{k}-y_{k}\rangle\geq\rho\|x_{k}-y_{k}\|^{\text{2}}.
\]
Hence
\[
\langle\nabla g(y_{k})-v,d_{k}\rangle =\langle u_k-v,y_{k}-x_k\rangle\leq-\rho\|d_{k}\|^{\text{2}},
\]
and the proof follows by combining the last inequality with \eqref{directionalDerivative}.

Finally, to prove (iii), if $d_{k}=0$ there is nothing to prove. Otherwise,  we have
$$
\lim_{ \lambda \downarrow 0} \frac{\phi(y_k+ \lambda d_k)-\phi(y_k)}{ \lambda}
= \phi'(y_k;d_k) \leq -\rho\|d_{k}\|^{2} < - \frac{\rho}{2} \|d_{k}\|^{2}<0.
$$
Hence, there is some $\widetilde{\lambda}_k > 0$ such that
$$
\frac{\phi(y_k+ \lambda d_k)-\phi(y_k)}{ \lambda}
\leq -\frac{\rho}{2} \|d_{k}\|^{2}, \quad \forall \lambda \in \left]0,\widetilde{\lambda}_k \right];
$$
that is
$$
\phi(y_k+ \lambda d_k) \leq \phi(y_k)-  \frac{\rho \lambda}{2} \|d_{k}\|^{2}, \quad \forall \lambda \in \left]0,\widetilde{\lambda}_k\right].
$$
Setting $\delta_k := \min\left\lbrace \widetilde{\lambda}_k, \frac{\rho}{2 \alpha}\right\rbrace$, we obtain
$$
\phi(y_k+ \lambda d_k) \leq \phi(y_k)-  \alpha \lambda^2 \|d_{k}\|^{2},\quad  \forall \lambda\in{]0,\delta_k]},
$$
which completes the proof.
\end{proof}

\begin{remark} (i) When the function $h$ is differentiable,
BDCA uses the same direction as the Mine--Fukushima algorithm~\cite{mine_minimization_1981}, since
$y_{k}+\lambda d_{k}=x_{k}+(1+\lambda)d_{k}$. The algorithm they propose is computationally undesirable in the sense that it uses an exact line search. This was later fixed in the Fukushima--Mine algorithm~\cite{fukushima_generalized_1981} by considering an Armijo type rule for choosing the step size
$$
x_{k+1}=x_{k}+\beta^{l}d_{k}=\beta^{l}y_{k}+\left(1-\beta^{l}\right)x_{k}
$$
for some $0<\beta<1$ and some nonnegative integer $l$. Since $0<\beta<1$, the step size $\lambda=\beta^{l}-1$ chosen by the Fukushima--Mine
algorithm~\cite{fukushima_generalized_1981} is always less than or
equal to zero, while in BDCA, only step sizes $\lambda\in{\big]0,\overline{\lambda}_k\big]}$ are explored. Also, the Armijo rule differs, as BDCA searches for some $\lambda_{k}$ such that
$\phi(y_{k}+\lambda_{k}d_{k})\leq\phi(y_{k})-\alpha\lambda_k^2\|d_k\|^2$, while the Fukushima--Mine algorithm requires $\phi(x_{k}+\beta^l d_{k})\leq\phi(x_{k})-\alpha\beta^l\|d_k\|^2$.\\
(ii) We know from~\cref{prop:main_inequality}
that
\[
\phi\left(y_{k}+\lambda d_{k}\right)\leq\phi(y_{k})-\alpha\lambda^2\|d_{k}\|^{2}\leq\phi(x_{k})-\left(\rho+\alpha\lambda^2\right)\|d_{k}\|^{2};
\]
thus, BDCA results in a larger decrease in the value of $\phi$
at each iteration than DCA.
As a result, we can expect BDCA to converge faster than DCA.
\end{remark}

\begin{example}[\cref{ex:function} revisited]\label{ex:1}
Consider again the function defined in \cref{ex:function} for $m=2$. The function $\phi$ can be expressed as a DC function of type~\eqref{eq:DC_2norm} with strongly convex terms by taking, for instance,
$$g(x,y)=\frac{3}{2}\left(x^2+y^2\right)+x+y\quad\text{and}\quad h(x,y)=|x|+|y|+\frac{1}{2}\left(x^2+y^2\right).$$
In \cref{fig:ex1a} we show the iterations generated by DCA and BDCA from the same starting point $(x_0,y_0)=(1,0)$, with $\alpha=0.1$, $\beta=0.5$ and $\overline{\lambda}_k=1$ for all $k$. Not only BDCA obtains a larger decrease than DCA in the value of $\phi$ at each iteration, but also the line search helps the sequence generated escape from the stationary point $(0,-1)$, which is not even a local minimum. As the function $h$ is not differentiable at $(x_0,y_0)$, there is freedom in the choice of the point in $\partial h(x_0,y_0)=\{2\}\times[-1,1]$ (we took the point $(2,0)$). In \cref{fig:ex1b} we plot the value of the function in the line search procedure of BDCA at the first iteration. The value $\lambda=0$ corresponds to the next iteration chosen by DCA, while BDCA choses $\lambda>0$, which permits to achieve an additional decrease in the value of $\phi$.

\begin{figure}[ht!]\centering
\subfigure[Iterations generated by DCA and BDCA\label{fig:ex1a}]{\includegraphics[width=0.49\textwidth]{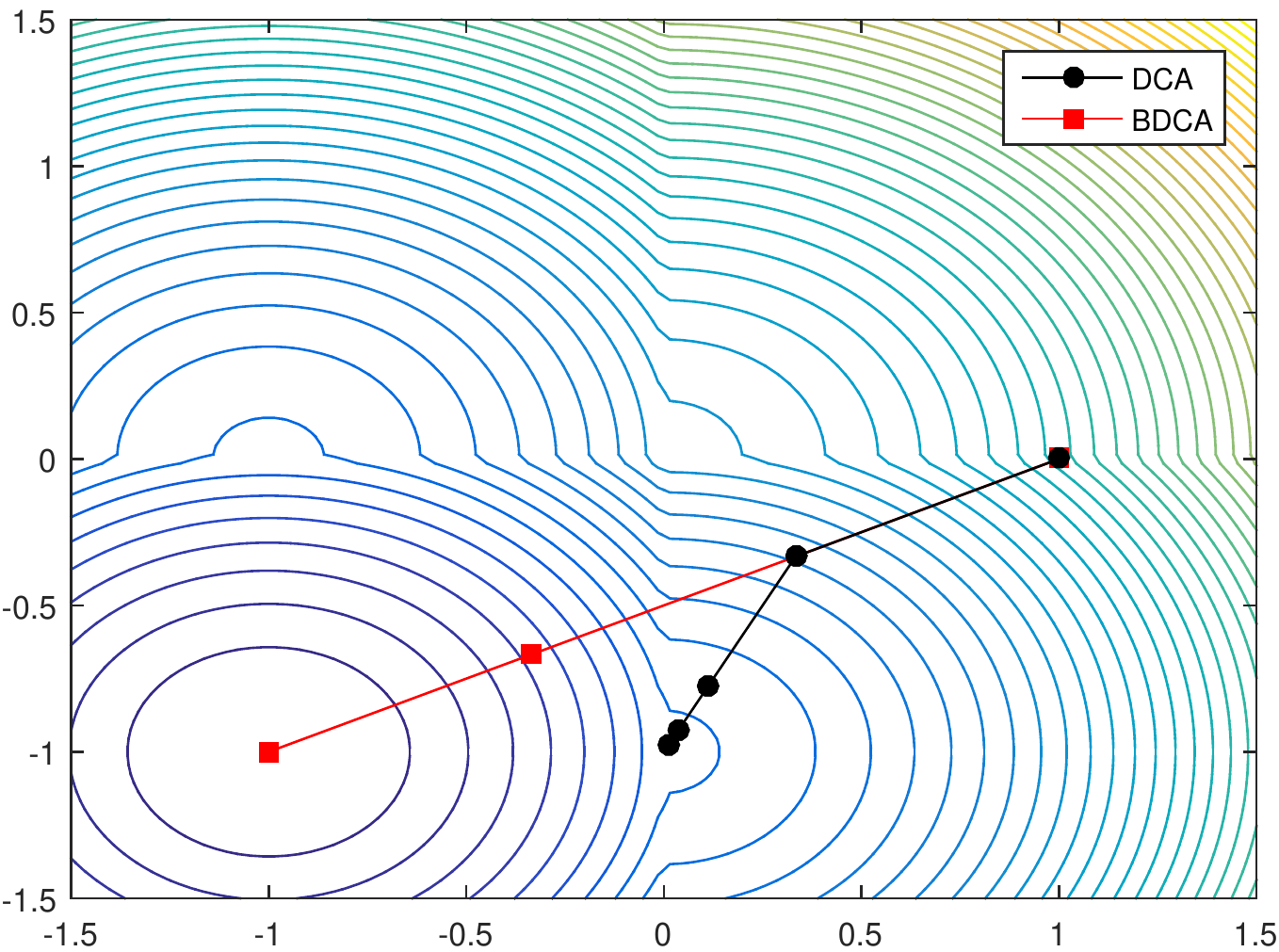}}
\subfigure[Line search of BDCA at the starting point\label{fig:ex1b}]{\includegraphics[width=0.49\textwidth]{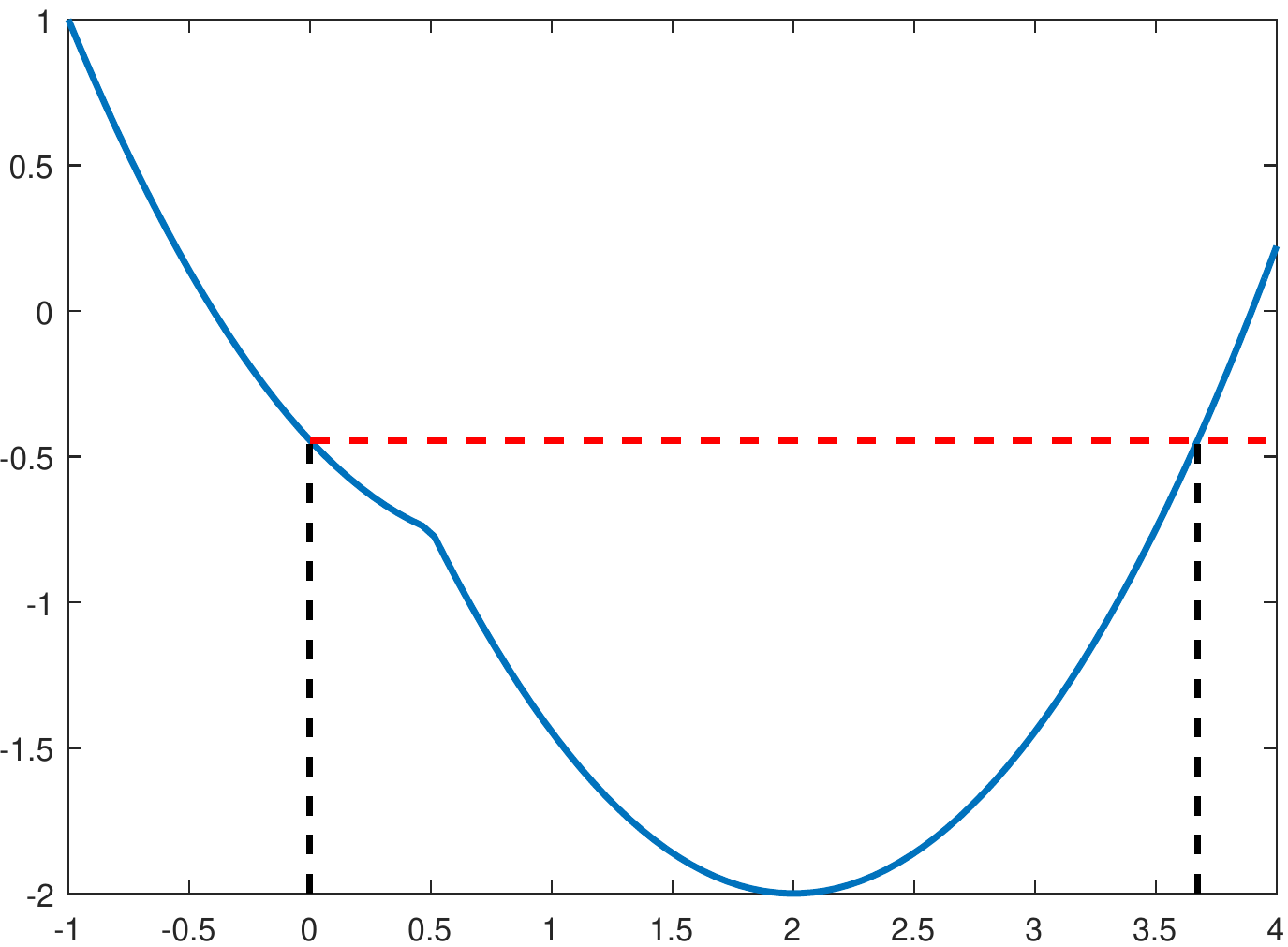}}
\caption{Illustration of~\cref{ex:1}}
\end{figure}
\begin{table}[ht!]\centering
\begin{tabular}{|c|c|c|c|c|}
\cline{2-5}
\multicolumn{1}{c|}{}
& $(-1,-1)$ & $(-1,0)$ & $(0,-1)$ & $(0,0)$\\
\hline
DCA & 249,763    &  249,841    &  250,204     & 250,192\\
BDCA & 996,104      &  1,922      &  1,974        &   0\\
\hline
\end{tabular}\caption{For one million random starting points in $[-1.5,1.5]^2$, we count the sequences generated by DCA and BDCA converging to each of the four stationary points}\label{tbl:example}
\end{table}
To demonstrate that, indeed, the line search procedure of BDCA helps the iterations escape from stationary points that are not critical points, we show in \cref{tbl:example} the results of running both algorithms for one million random starting points. For only 25\% of the starting points, DCA finds the optimal solution, while BDCA finds it in 99.6\% of the instances.\qede
\end{example}

The next example complements the one given in~\cite[Remark~1]{BDCA2018}. It shows that the direction used by BDCA can be an ascent direction at $y_k$ even when this point is not the global minimum of $\phi$. Thus, \cref{prop:main_inequality} does not remain valid when $g$ is not differentiable, and the scheme cannot be further extended.

\begin{example}[Failure of BDCA when $g$ is not differentiable]\label{ex:failure}
Consider now the following modification of the previous example
$$g(x,y)=-\frac{5}{2}x+x^2+y^2+|x|+|y|\quad\text{and}\quad h(x,y)=\frac{1}{2}\left(x^2+y^2\right),$$
so that now $h$ is differentiable but $g$ is not.
Let $(x_0,y_0)=\left(\frac{1}{2},1\right)$. Then, the next point generated by DCA is $(x_1,y_1)=(1,0)$ and $d_0:=(x_1,y_1)-(x_0,y_0)=\left(\frac{1}{2},-1\right)$ is not a descent direction for $\phi$ at $(x_1,y_1)$. Indeed, one can easily check that
$$\phi'((x_1,y_1);d_0)=\lim_{t\downarrow 0}\frac{\phi\left((1,0)+t\left(\frac{1}{2},-1\right)\right)-\phi(1,0)}{t}=\frac{3}{4},$$
see \cref{fig:failure}. Actually, it holds that
$$\phi\left((x_1,y_1)+td_0\right)-\phi(x_1,y_1)=\frac{5t^2}{8} + \frac{3t}{4},$$
so $\phi\left((x_1,y_1)+td_0\right)>\phi(x_1,y_1)$ for all $t>0$.
\begin{figure}[ht!]\centering
\subfigure[Iterations generated by DCA and search direction of BDCA at $(1,0)$\label{fig:ex2a}]{\includegraphics[width=0.49\textwidth]{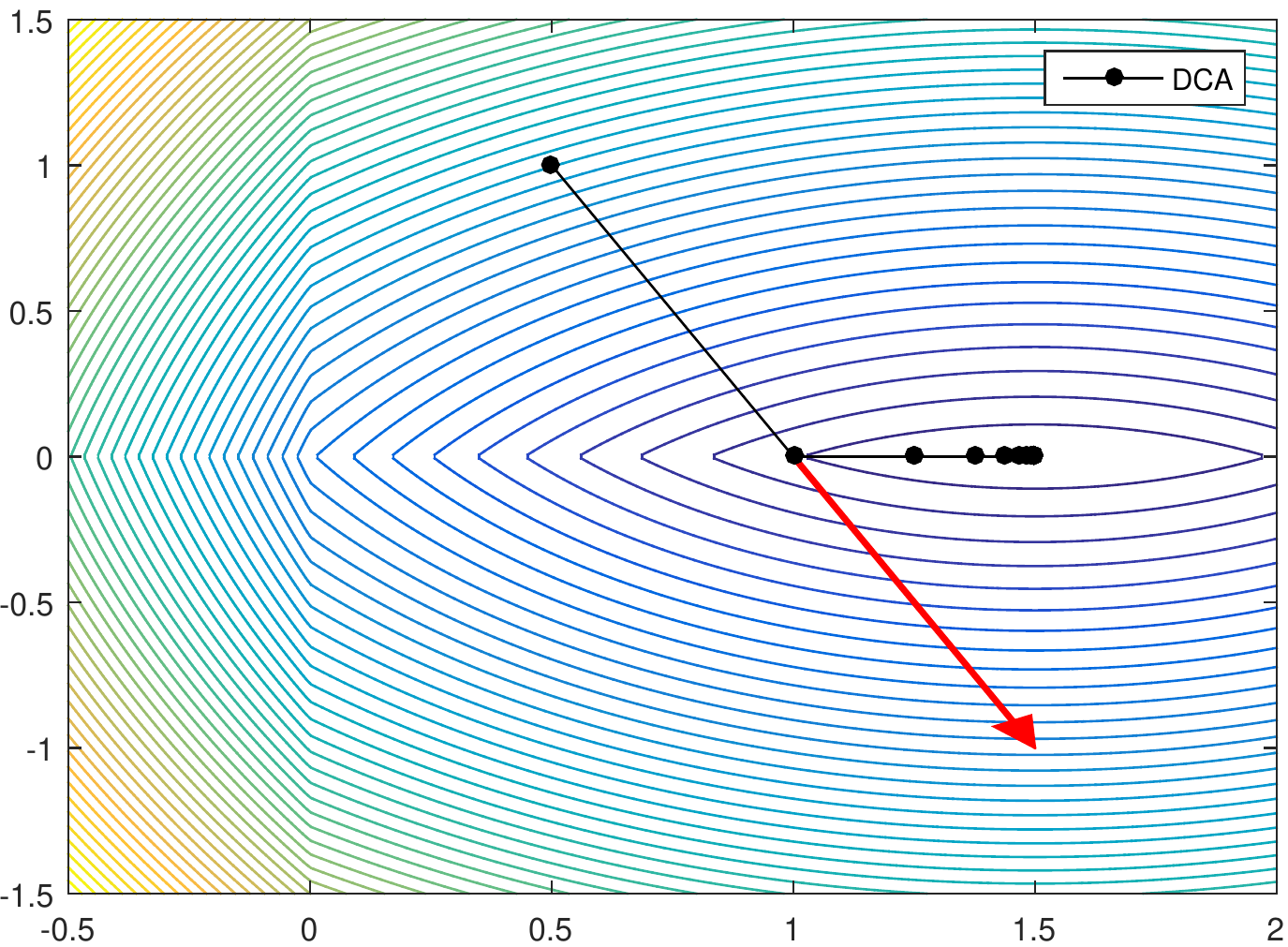}}
\subfigure[Line search of BDCA at the point (1,0)\label{fig:ex2b}]{\includegraphics[width=0.49\textwidth]{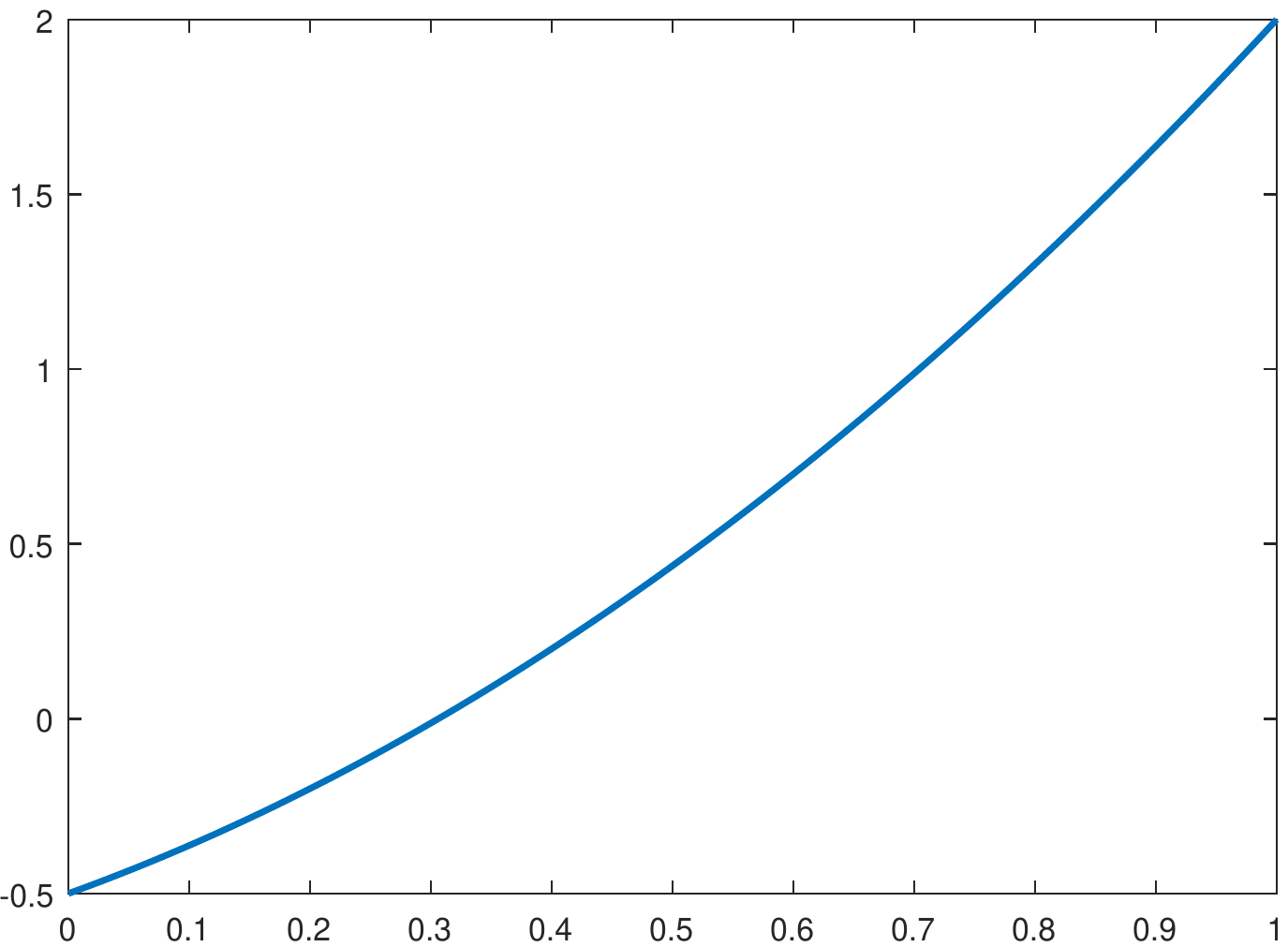}}
\caption{Illustration of~\cref{ex:failure}\label{fig:failure}}
\end{figure}

In contrast with the example in~\cite[Remark~1]{BDCA2018}, observe that here $(x_1,y_1)$ is not the global minimum of $\phi$. In fact, the iterates generated by DCA converge to the global minimum of $\phi$, as shown in \cref{fig:ex2a}.\qede
\end{example}

As proved next, the failure of BDCA shown in \cref{ex:failure} can only occur for~$n\geq 2$.
\begin{proposition}
Let $\phi=g-h$, where $g:\mathbb{R}\to\mathbb{R}$ and $h:\mathbb{R}\to\mathbb{R}$ are convex and $h$ is differentiable. If $h'(x)\in\partial g(y)$ and $0\not\in\partial_C \phi(y)$, then $\phi'(y;y-x)<0$.
\end{proposition}

\begin{proof}
First, observe that
$$\phi'(y;y-x)=(y-x)\sup_{z\in\partial g(y)}\left\{z-h'(y)\right\}.$$
Since $h$ is convex, one has
$$\left(h'(x)-h'(y)\right)(x-y)\geq 0.$$
Suppose that $x-y>0$. Then, $h'(x)\geq h'(y)$. Since $h'(y)\not\in\partial g(y)$ and $\partial g(y)$ is convex, we deduce that $h'(y)<z$ for all $z\in\partial g(y)$, which implies $\phi'(y;y-x)< 0$. A similar argument shows that $\phi'(y;y-x)< 0$ when $x-y<0$. This concludes the proof.
\end{proof}

We are now in a position to state our first convergence result of the iterative sequence generated by BDCA, whose statement coincides with~\cite[Proposition~5]{BDCA2018}. The first part of its proof requires some small adjustments due to the nonsmoothness of $h$.

\begin{theorem}\label{prop:innerloop}For any $x_{0}\in\mathbb{R}^{m}$, either BDCA
returns a critical point of~$\left(\mathcal{P}\right)$ or it generates
an infinite sequence such that the following holds.
\begin{enumerate}
\item $\phi(x_{k})$ is monotonically decreasing and convergent to some
$\phi^{*}$.
\item Any limit point of $\{x_{k}\}$ is a critical point of $\left(\mathcal{P}\right)$.
If in addition, $\phi$ is coercive then there exits a subsequence
of $\{x_{k}\}$ which converges to a critical point of~$\left(\mathcal{P}\right)$.
\item $\sum_{k=0}^{+\infty}\|d_{k}\|^{2}<+\infty.$ Further, if there is some $\overline{\lambda}$ such that $\lambda_k\leq\overline{\lambda}$ for all $k$, then $\sum_{k=0}^{+\infty}\|x_{k+1}-x_{k}\|^{2}<+\infty$.\end{enumerate}
\end{theorem}
\begin{proof}
If BDCA stops at Step~3 and returns $x_{k}$, then $x_{k}=y_k$.
Because $y_k$ is the unique solution of the strongly convex problem $(\mathcal{P}_k)$, we have
$$\nabla g(x_k)=u_k \in \partial h(x_k),
$$
i.e., $x_k$ is a critical point
of~$\left(\mathcal{P}\right)$.
Otherwise, by \cref{prop:main_inequality}
and Step~4 of BDCA, we have
\begin{equation}\label{eq:phi_decreasing}
\phi(x_{k+1})\leq\phi(y_{k})-\alpha\lambda^2_{k}\|d_{k}\|^{2}\leq\phi(x_{k})-\left(\alpha\lambda^2_{k}+\rho\right)\|d_{k}\|^{2}.
\end{equation}
Therefore, the sequence $\{\phi(x_{k})\}$ converges
to some $\phi^{*}$, since is monotonically decreasing
and bounded from below by~\eqref{eq:phi_bounded_below}. This proves~\emph{(i)}.
As a consequence, we obtain
\[
\phi(x_{k+1})-\phi(x_{k})\to0,
\]
which implies $\|d_{k}\|^{2}=\|y_{k}-x_{k}\|^{2}\to 0$, by~\eqref{eq:phi_decreasing}.

If $\bar{x}$ is a limit point of $\{x_{k}\}$, there exists a subsequence $\{x_{k_{i}}\}$
converging to $\bar{x}$. Then, as $\|y_{k_{i}}-x_{k_{i}}\|\to 0$,
we have $y_{k_{i}}\to\bar{x}$.
Since $\nabla g$ is continuous, we get
$$
u_{k_i} = \nabla g(y_{k_i}) \to \nabla g(\bar{x}).
$$
Hence, we deduce $\nabla g(\bar{x})\in \partial h(\bar{x})$, thanks to the closedness of the graph of $\partial h$ (see~\cite[Theorem 24.4]{TR}). When $\phi$ is coercive, by~\emph{(i)},
the sequence $\{x_{k}\}$ must be bounded, which implies the rest of the claim in~\emph{(ii)}.

The proof of~\emph{(iii)} is similar to that of~\cite[Proposition~5(iii)]{BDCA2018} and
is thus omitted.
\end{proof}

\begin{remark}
In our approach, both functions $g$ and $h$ are assumed to be strongly convex with constant $\rho>0$. It is well-known that the performance of DCA heavily depends on the decomposition of the objective function~\cite{an_numerical_1996,tao1998dc}.
There is an infinite number of ways of doing this and it is challenging to find a ``good'' one~\cite{tao1998dc}.
To get rid of this assumption, one could add a proximal term $\frac{\rho_k}{2}\|x-x_k\|^2$
to the objective of the convex optimization subproblem~$(\mathcal{P}_k)$ in Step 2, as done in the proximal point algorithm (see~\cite{fukushima_generalized_1981}). This technique is employed in the proximal DCA, see~\cite{AnNam2017,BB2018,An2018,moudafi_proximalDC_2006}.
With some minor adjustments in the proofs, it is easy to show that the resulting algorithm satisfies both \cref{prop:main_inequality,prop:innerloop}.
\end{remark}

\section{Convergence  under the Kurdyka--\L{}ojasiewicz property}\label{sec:KL}

In this section, we  prove the convergence of the sequence generated by BDCA as long as the sequence has a cluster point at which $\phi$ satisfies the strong Kurdyka--\L{}ojasiewicz inequality~\cite{lojasiewicz1965ensembles,Kurdyka,bolteArisLewis2007} and $\nabla g$ is locally Lipschitz. As we shall see, under some additional assumptions, linear convergence can be also guaranteed.

\begin{definition}
Let $f: \mathbb{R}^{m} \to \mathbb{R} $ be a locally Lipschitz function.
We say that $f$ satisfies the \emph{strong Kurdyka--\L{}ojasiewicz inequality} at $x^* \in \mathbb{R}^m$ if
 there exist $\eta \in {]0, +\infty[}$, a neighborhood $U$ of
$x^*$, and a concave function $\varphi : [0,\eta] \to {[0,+\infty[}$ such that:
\begin{enumerate}
\item $\varphi (0)=0$;
\item $\varphi$ is of class $\mathcal{C}^1$ on ${]0,\eta[}$;
\item $\varphi' > 0$  on $]{0,\eta[}$;
\item for all
$x \in U$ with $f(x^*) < f(x) < f(x^*)+\eta$
we have
$$
\varphi'(f(x)-f(x^*)) \operatorname{dist}\left(0,\partial_C f(x)\right) \geq 1.
$$
\end{enumerate}
\end{definition}

For strictly differentiable functions the latter reduces to the standard definition of the K\L{}-inequality.
Bolte et al. \cite[Theorem~14]{bolteArisLewis2007} show that {\em definable functions} satisfy the strong K\L{}-inequality
at each point in $\dom \partial_C f$, which covers a large variety of practical cases.

\begin{remark}\label{rem:Aris}
Although the concavity of the function $\varphi$ does not explicitly appear in the statement of \cite[Theorem~14]{bolteArisLewis2007}, the function $\varphi$ can be chosen to be concave (since $\varphi$ is \mbox{o-minimal} by construction, its second derivative exists and maintains the sign on an interval ${]0,\delta[}$, and this sign is necessarily negative). If the function $f$ is not o-minimal but is convex and satisfies the Kurdyka--\L{}ojasiewicz inequality with a function $\varphi$ which is not concave, then $f$ also satisfies the Kurdyka--\L{}ojasiewicz inequality with another function $\Psi$ which is concave (see~\cite[Theorem~29]{BDLM10}).
\end{remark}

\begin{theorem}
\label{th:convergence_KL}
For any $x_{0}\in\mathbb{R}^{m}$, consider the sequence $\left\{ x_{k}\right\} $ generated by the BDCA.
Suppose that $\{x_{k}\}$ has a cluster point~$x^{*}$, that $\nabla g$ is locally
Lipschitz continuous around $x^{*}$
and that $\phi$ satisfies the strong Kurdyka--\L{}ojasiewicz inequality at $x^{*}$.
Then $\{x_{k}\}$  converges to~$x^{*},$ which is a critical point
of~$\left(\mathcal{P}\right)$.
\end{theorem}
\begin{proof}
By \cref{prop:innerloop}, we have
$\lim_{k\to+\infty}\phi(x_{k})=\phi^{*}.$
Let $x^{*}$ be a cluster point of the sequence~$\left\{ x_{k}\right\} $. Then,
there exists a subsequence $\{x_{k_{i}}\}$ of $\{x_{k}\}$ such that $\lim_{i\to+\infty}x_{k_i}=x^*$. Thanks to the continuity of $\phi$, we deduce
\[
\phi(x^{*})=\lim_{i\to+\infty}\phi(x_{k_{i}})=\lim_{k\to\infty}\phi(x_{k})=\phi^{*}.
\]
Hence, the function $\phi$ is finite and has the same value~$\phi^{*}$ at every
cluster point of $\{x_{k}\}$.

If $\phi(x_{k})=\phi^{*}$ for some
$k>1$, then $\phi(x_{k})=\phi(x_{k+1})$, because the sequence $\{\phi(x_{k})\}$ is decreasing. From~\eqref{eq:phi_decreasing}, we deduce that $d_k=0$, so BDCA terminates after a finite number
of steps. Thus, from now on, we assume that $\phi(x_{k})>\phi^{*}$ for
all $k$.

Since $\nabla g$ is locally Lipschitz around $x^{*}$, there
exist some constants $L\geq0$ and $\delta_1>0$ such that
\begin{equation}
\|\nabla g(x)-\nabla g(y)\|\leq L\|x-y\|,\quad\forall x,y\in\mathbb{B}(x^{*},\delta_1).\label{eq:nabla_g_Lip}
\end{equation}

Further, since  $\phi$ satisfies the strong Kurdyka--\L{}ojasiewicz inequality at $x^{*}$,
there exist $\eta \in {]0, +\infty[}$, a neighborhood $U$ of
$x^*$, and a continuous and concave function $\varphi : [0,\eta] \to {[0,+\infty[}$ such that
for every $x \in U$ with $\phi(x^*) < \phi(x) < \phi(x^*)+\eta$, we have
\begin{equation} \label{KL}
\varphi'(\phi(x)-\phi(x^*)) \operatorname{dist}\left(0,\partial_C \phi(x)\right) \geq 1.
\end{equation}
Take $\delta_2$ small enough that $\mathbb{B}(x^*,\delta_2) \subset U$ and set $\delta := \frac{1}{2} \min \left\lbrace \delta_1,\delta_2 \right\rbrace $. Let
\begin{equation}\label{eq:K}
K:=\max_{\lambda\geq 0}\frac{L(1+\lambda)}{\alpha \lambda^2+\rho},
\end{equation}
which is attained at $\hat\lambda=-1+\sqrt{1+\rho/\alpha}$. Since $\lim_{i\to+\infty}x_{k_{i}}=x^{*}$, $\lim_{i\to+\infty}\phi(x_{k_{i}})=\phi^{*}$, $\phi(x_{k})>\phi^{*}$ for
all $k$, and $\varphi$ is continuous, we can find an index $N$ large enough such that
\begin{equation}\label{eq:BallCondition}
x_N \in \mathbb{B}(x^*,\delta), \quad \phi^* < \phi(x_N) < \phi^*+\eta
\end{equation}
and
\begin{equation}
\|x_{N}-x^{*}\|+ K  \varphi \left(\phi(x_{N})-\phi^*\right) < \delta.
\label{eq:BallCondition2}
\end{equation}
By \cref{prop:innerloop}(iii), we know that $d_{k}=y_{k}-x_{k}\to0$.
Then, taking a larger $N$ if needed, we can ensure that
\[
\|y_{k}-x_{k}\|\leq\delta,\quad\forall k\geq N.
\]
For all $k \geq N$ such that
$x_k \in \mathbb{B}(x^*,\delta)$, we have
\[
\|y_{k}-x^{*}\|\leq\|y_{k}-x_{k}\|+\|x_{k}-x^{*}\|\leq2\delta\leq \delta_1,
\]
then, using~\eqref{eq:nabla_g_Lip}, we obtain
$$
\|\nabla g(y_k) -\nabla g(x_k)\| \leq L\|y_k-x_k\| = \frac{L}{1+\lambda_k} \|x_{k+1}-x_k\|.
$$
On the other hand, we have
from the optimality condition of~$(\mathcal{P}_k)$ that
$$
\nabla g(y_k) = u_k \in \partial h(x_k),
$$
which implies, by \cref{Clarkesumrule},
$$
\nabla g(y_k) -\nabla g(x_k)\in \partial h(x_k) -\nabla g(x_k)=  \partial_C \left(-\phi(x_k)\right)=-\partial_C \phi(x_k).
$$
Therefore,
\begin{equation}
\operatorname{dist}\left(0,\partial_C \phi(x_k)\right) \leq \|\nabla g(y_k) -\nabla g(x_k)\| \leq  \frac{L}{1+\lambda_k} \|x_{k+1}-x_k\|.
\label{KL2}
\end{equation}

For all $k \geq N$ such that
$x_k \in \mathbb{B}(x^*,\delta)$ and $\phi^* < \phi(x_k) < \phi^*+\eta$,
 it follows from \eqref{KL2}, the concavity of $\varphi$, \eqref{KL} and \eqref{eq:phi_decreasing} that
\begin{align*}
\frac{L}{1+\lambda_k} \|x_k&-x_{k+1}\| \left( \varphi \left(\phi(x_k)-\phi^*\right) -\varphi \left(\phi(x_{k+1})-\phi^*\right)\right)\\
&\geq \operatorname{dist}\left(0,\partial_C \phi(x_k)\right) \left( \varphi \left(\phi(x_k)-\phi^*\right) -\varphi \left(\phi(x_{k+1})-\phi^*\right)\right)\\
&\geq \operatorname{dist}\left(0,\partial_C \phi(x_k)\right)  \varphi' \left(\phi(x_k)-\phi^*\right)\left( \phi(x_{k})-\phi(x_{k+1})\right)\\
&\geq  \phi(x_{k})-\phi(x_{k+1})\\
&\geq  \left(\alpha \lambda_{k}^2+\rho \right)\|y_k-x_k\|^2
= \frac{\alpha \lambda_{k}^2+\rho}{(1+\lambda_{k})^2}\|x_k-x_{k+1}\|^2,
\end{align*}
which implies, by~\eqref{eq:K}, that
\begin{align}\label{eq:StepRelation2}
\|x_k-x_{k+1}\| \nonumber
&\leq \frac{L(1+\lambda_k)}{\alpha \lambda_{k}^2+\rho} \left( \varphi \left(\phi(x_k)-\phi^*\right) -\varphi \left(\phi(x_{k+1})-\phi^*\right)\right)\\
&\leq K \left( \varphi \left(\phi(x_k)-\phi^*\right) -\varphi \left(\phi(x_{k+1})-\phi^*\right)\right).
\end{align}

We prove by induction that $x_{k}\in\mathbb{B}(x^{*},\delta)$
for all $k\geq N$. Indeed, from~\eqref{eq:BallCondition} the claim
holds for $k=N$. We suppose that it also holds for $k=N,N+1,\ldots,N+p-1$,
with~$p\geq1$.
Since $\left\lbrace \phi(x_k) \right\rbrace $ is a decreasing sequence converging to $\phi^*$, our choice of $N$
implies that $\phi^* < \phi(x_k) < \phi^*+\eta$ for all $k\geq N$.
Then~\eqref{eq:StepRelation2} is valid for $k=N,N+1,\ldots,N+p-1$.
Hence,
\begin{align*}
\left\Vert x_{N+p}-x^{*}\right\Vert
& \leq\left\Vert x_{N}-x^{*}\right\Vert +\sum_{i=1}^{p}\left\Vert x_{N+i}-x_{N+i-1}\right\Vert \\
 & \leq\left\Vert x_{N}-x^{*}\right\Vert +K\sum_{i=1}^{p}\left[\varphi \left(\phi(x_{N+i-1})-\phi^*\right) -\varphi \left(\phi(x_{N+i})-\phi^*\right)\right]\\
  & \leq\left\Vert x_{N}-x^{*}\right\Vert +K\varphi \left(\phi(x_{N})-\phi^*\right) < \delta,
\end{align*}
 where the last inequality follows from~\eqref{eq:BallCondition2}.

Thus, adding~\eqref{eq:StepRelation2} from $k=N$ to $P$, we get
\begin{equation*}
\sum_{k=N}^{P}\|x_{k+1}-x_{k}\| \leq
 K  \varphi \left(\phi(x_N)-\phi^*\right), \label{eq:bound_finite_sum}
\end{equation*}
and taking the limit as $P\to+\infty$, we conclude that
\begin{equation}
\sum_{k=1}^{+\infty}\|x_{k+1}-x_{k}\|<+\infty.\label{eq:finite_sum}
\end{equation}
Therefore, $\{x_{k}\}$ is a Cauchy sequence, and since~$x^{*}$
is a cluster point of~$\{x_{k}\}$, the whole sequence
converges to~$x^{*}$. By \cref{prop:innerloop}, $x^{*}$
must be a critical point of~$\left(\mathcal{P}\right)$.
\end{proof}

\begin{remark}(i) Observe that \cref{th:convergence_KL} also holds under the assumption that $-\phi$ satisfies the \emph{Kurdyka--\L{}ojasiewicz inequality} (which is the same estimate but for the limiting subdifferential).

(ii) As mentioned before, if one sets $\overline{\lambda}_k=0$ for all $k$, then BDCA becomes DCA. In this case, \cref{th:convergence_KL} is akin to~\cite[Theorem~3.4]{JOTA2018}, where the function $\phi$ is assumed to be subanalytic. We also note that in this setting only one of the functions $g$ or $h$ needs to be strongly convex, since one can easily check that~\cite[Proposition~3]{BDCA2018} still holds, and \cref{prop:main_inequality}(ii) is not needed anymore.
\end{remark}

Next, we establish the convergence rate on the iterative sequence $\left\lbrace x_k\right\rbrace$ when
$\phi$ satisfies the  Kurdyka--\L{}ojasiewicz inequality with $\varphi(t)=M t^{1-\theta}$
 for some $M>0$ and $0\leq \theta<1$. Observe that this property holds for all globally subanalytic functions \cite[Corollary~16]{bolteArisLewis2007}, which covers many classes of functions in applications.
We will employ the following useful lemma, whose proof appears within that of~\cite[Theorem~2]{attouch2009convergence} for specific values of $\alpha$ and $\beta$.
\begin{lemma} \cite[Lemma 1]{BDCA2018}
\label{lem:rate_convergence}Let~$\left\{ s_{k}\right\} $ be a nonnegative sequence
in~$\mathbb{R}$ and let~$\alpha,\beta$ be some positive constants.
Suppose that $s_{k}\to0$ and that the sequence satisfies
\begin{equation*}
s_{k}^{\alpha}\leq\beta(s_{k}-s_{k+1}),\quad\text{for all }k\text{ sufficiently large.}\label{eq:ineq_seq}
\end{equation*}
Then,
\begin{enumerate}
\item if~$\alpha=0$, the sequence~$\left\{ s_{k}\right\} $ converges
to~$0$ in a finite number of steps;
\item if~$\alpha\in\left]0,1\right]$, the sequence~$\left\{ s_{k}\right\} $
converges linearly to~$0$ with rate~$1-\frac{1}{\beta}$;
\item if~$\alpha>1$, there exists~$\eta>0$ such that
\[
s_{k}\leq\eta k^{-\frac{1}{\alpha-1}},\quad\text{for all }k\text{ sufficiently large.}
\]
\end{enumerate}
\end{lemma}

\begin{theorem}
\label{th:convergence_rate_KL}
Suppose that the sequence $\{x_{k}\}$ generated by the BDCA has the limit point~$x^{*}$.
Assume that $\nabla g$ is locally Lipschitz continuous around $x^{*}$
and $\phi$ satisfies the strong Kurdyka--\L{}ojasiewicz inequality at $x^{*}$ with $\varphi(t)=M t^{1-\theta}$
 for some $M>0$ and $0\leq \theta<1$.
Then, the following convergence rates are guaranteed:
\begin{enumerate}
\item if $\theta=0$, then the sequence~$\{x_{k}\}$ converges in a finite number of steps to~$x^{*}$;
\item if $\theta\in\left]0,\frac{1}{2}\right]$, then the sequence~$\{x_{k}\}$
converges linearly to~$x^{*}$;
\item if $\theta\in\left]\frac{1}{2},1\right[$, then there exist a positive
constant $\eta$ such that
\begin{gather*}
\|x_{k}-x^{*}\|\leq\eta k^{-\frac{1-\theta}{2\theta-1}}
\end{gather*}
for all large~$k$. \end{enumerate}
\end{theorem}
\begin{proof}
By~\eqref{eq:finite_sum}, we know that $s_{i}:=\sum_{k=i}^{+\infty}\|x_{k+1}-x_{k}\|$ is finite. Since~$\|x_{i}-x^{*}\|\leq s_{i}$
by the triangle inequality, the rate of convergence of~$x_{i}$
to $x^{*}$ can be deduced from the convergence rate of $s_{i}$ to~0.

Adding~\eqref{eq:StepRelation2} from~$i$ to~$P$ with $N\leq i\leq P$,
we have
$$\sum_{k=i}^{P}\|x_{k+1}-x_{k}\|\leq K \varphi\left(\phi(x_{i})-\phi^{*}\right)
=K M (\phi(x_{i})-\phi^{*})^{1-\theta},
$$
which implies that
\begin{equation} \label{relation1}
s_{i}=\lim_{P\to+\infty}\sum_{k=i}^{P}\|x_{k+1}-x_{k}\|\leq K M (\phi(x_{i})-\phi^{*})^{1-\theta}.
\end{equation}
Since $\phi$ satisfies the strong Kurdyka--\L{}ojasiewicz inequality at $x^{*}$ with $\varphi(t)=M t^{1-\theta}$, we have
$$
M(1-\theta)\left(\phi(x_{i})-\phi^{*}\right)^{-\theta}\operatorname{dist}\left(0,\partial_C\phi(x_i)\right) \geq 1.
$$
This and \eqref{KL2} imply
\begin{align}\label{relation2}
\left(\phi(x_{i})-\phi^{*}\right)^{\theta} \nonumber
&\leq M(1-\theta)\operatorname{dist}\left(0,\partial_C\phi(x_i)\right)\\\nonumber
&\leq \frac{ML(1-\theta)}{1+\lambda_i} \|x_{i+1}-x_i\| \\
&\leq ML(1-\theta)\|x_{i+1}-x_i\|.
\end{align}
Combining \eqref{relation1} and \eqref{relation2}, we obtain
\begin{equation*} 
s_{i}^{\frac{\theta}{1-\theta}} \leq \left(K M\right)^{\frac{\theta}{1-\theta}} (\phi(x_{i})-\phi^{*})^{\theta}\leq ML(1-\theta)\left(K M\right)^{\frac{\theta}{1-\theta}}(s_{i}-s_{i+1}).
\end{equation*}
Applying \cref{lem:rate_convergence}, with~$\alpha:=\frac{\theta}{1-\theta}$
and~$\beta:=ML(1-\theta)\left(K M\right)^{\frac{\theta}{1-\theta}}$, we deduce the convergence rates in~\emph{(i)}-\emph{(iii)}.
\end{proof}

\section{Applications and Numerical Experiments}\label{sec:numa}
The purpose of this section is to numerically compare the performance of DCA and BDCA. All our codes were written in Python~2.7 and the tests were run on an Intel Core i7-4770 CPU \@3.40GHz with 32GB RAM, under Windows 10 (64-bit).

In all the experiments in this section we use the following strategy for choosing the trial step size in Step~4 of BDCA, which makes use of the previous step sizes. We emphasize that the convergence results in the previous sections apply to any possible choice of the trial step sizes $\overline{\lambda}_k$. This is in contrast with~\cite{BDCA2018}, where $\overline{\lambda}_k$ had to be chosen constantly equal to some fixed parameter $\overline{\lambda}>0$.\medskip
\begin{center}
\bgroup 	\renewcommand\theenumi{\arabic{enumi}.} 	 \renewcommand\labelenumi{\theenumi}
\fbox{%
\begin{minipage}[t]{.98\textwidth}%
\textbf{Self-adaptive trial step size}\medskip\\
Fix $\gamma>1$. Set $\overline{\lambda}_0=0$. Choose some $\overline{\lambda}_1>0$ and obtain $\lambda_1$ by Step~4 of BDCA. \\ For any $k\geq 2$:
\begin{enumerate}
\item IF $\lambda_{k-2}=\overline{\lambda}_{k-2}$ AND $\lambda_{k-1}=\overline{\lambda}_{k-1}$ THEN set $\overline{\lambda}_{k}:=\gamma\lambda_{k-1}$;
    ELSE set $\overline{\lambda}_{k}:=\lambda_{k-1}$.
\item Obtain $\lambda_{k}$ from $\overline{\lambda}_{k}$ by Step~4 of BDCA.
\medskip{}
 \end{enumerate}
\end{minipage}}
\egroup{}
\end{center}\medskip

The latter \emph{self-adaptive strategy} uses the step size that was chosen in the previous iteration as a new trial step size for the next iteration, except in the case where two consecutive trial step sizes were successful. In that case, the trial step size is increased by multiplying the previously accepted step size by $\gamma>1$. Thus, we used a somehow conservative strategy in our experiments, where two successful iterations are needed before increasing the trial step size. Other strategies could be easily considered. Since we set $\overline{\lambda}_0=0$, the first iteration is computed with DCA. In all our experiments we took $\gamma:=2$.

The self-adaptive strategy for the trial step size has two key advantages with respect to the \emph{constant strategy} $\overline{\lambda}_k=\overline{\lambda}>0$, which was used in~\cite{BDCA2018}. The most important one is that we observed in our numerical tests almost a two times speed up in the running time of BDCA. The second advantage is that it is more adaptive and less sensitive to a wrong choice of the parameters. Indeed, in the constant strategy, a very large value of $\overline{\lambda}$ could make BDCA slow, due to the internal iterations needed in the backtracking step. On the other hand, a small value of $\overline{\lambda}$ would provide a trial step size that will be readily accepted, but will result in a small advantage of BDCA against DCA.

In the next two subsections, we compare the performance of DCA and BDCA in two important nonsmooth problems in data analysis: the Minimum Sum-of-Squares Clustering problem and the Multidimensional Scaling problem. Before doing that, let us begin by numerically demonstrating that the self-adaptive strategy permits to further improve the results of BDCA in the smooth problem arising from the study of systems of biochemical reactions tested in~\cite{BDCA2018}, where BDCA was shown to be more than four times faster than DCA. To this aim, we used the same setting than in~\cite[Section~5]{BDCA2018}. For each of five randomly selected starting points, we obtained the $1000^\text{th}$ iterate of BDCA with constant trial step size strategy $\overline{\lambda}=50$. Next, both BDCA with self-adaptive strategy (with $\beta=0.1$) and DCA were run from the same starting point until they reached the same objective value as the one obtained by BDCA with constant strategy. Instead of presenting a table with the results, we show in \cref{fig:ratios_MP} the ratios of the running times between the three algorithms, which permits to readily compare the three algorithms. On average, BDCA with self-adaptive strategy was $6.7$ times faster than DCA, and was $1.7$ times faster than BDCA with constant strategy, which in turns was $4.2$ times faster than DCA.

\begin{figure}[ht]
\centering
\includegraphics[width=.75\textwidth]{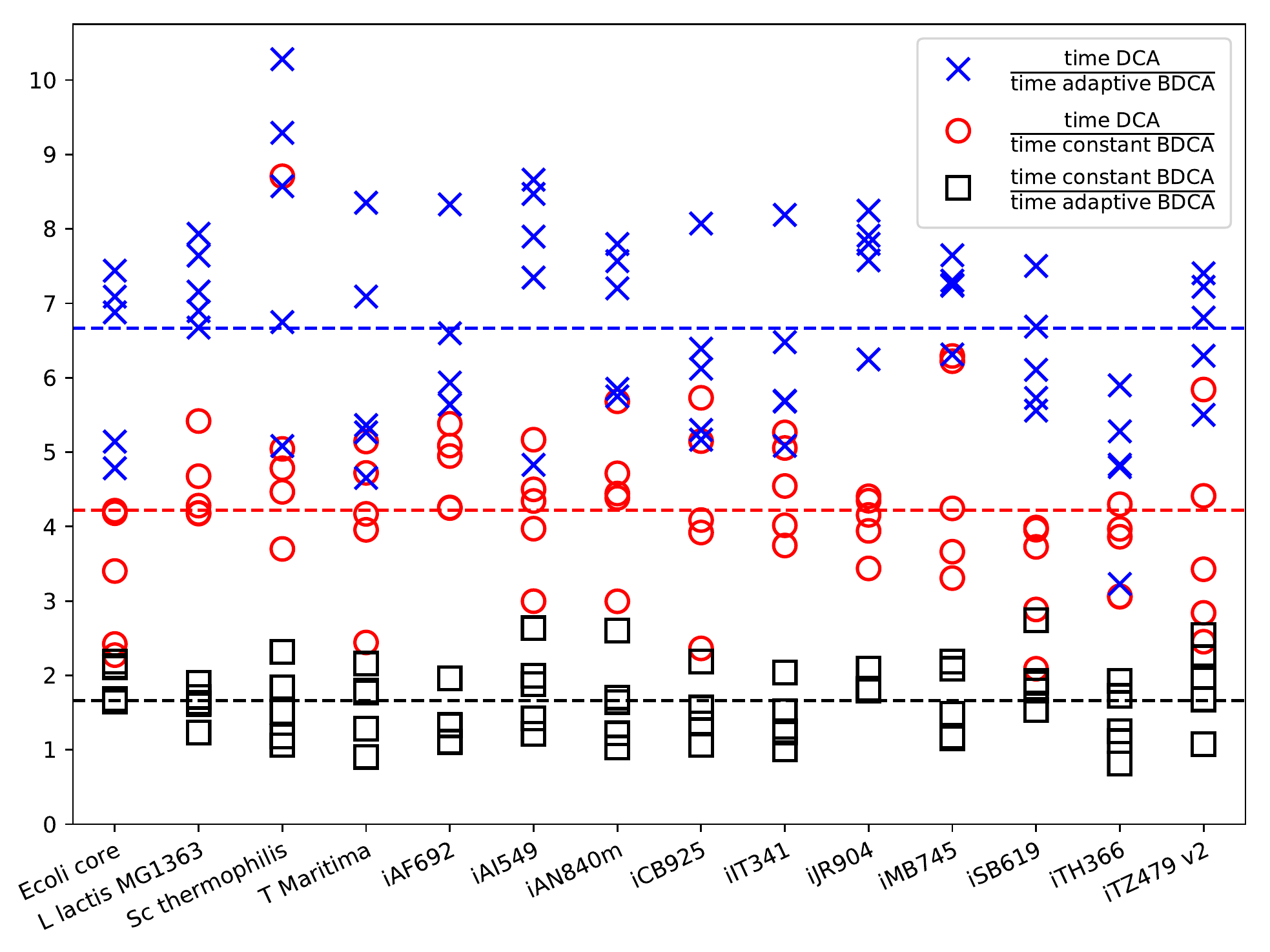}
\caption{Ratios of the running times of DCA, BDCA with constant trial step size and BDCA with self-adaptive trial step size for finding a steady state of various biochemical reaction network models~\cite{BDCA2018}. For each of the models, the algorithms were run using the same five random starting points. The average is represented with a dashed line.}\label{fig:ratios_MP}
\end{figure}

In the next two subsections we present various experiments with problems in data analysis. We consider two types of data: real and random. As real data, we use the geographic coordinates of the Spanish cities with more than 500 habitants\footnote{The data can be retrieved from the Spanish National Center of Geographic Information at \href{http://centrodedescargas.cnig.es}{http://centrodedescargas.cnig.es}.}. The advantage of this relatively large data in $\R^2$ is that it permits to visually illustrate some of the experiments.

\subsection {The Minimum Sum-of-Squares Clustering Problem} \label{SSCP}
Clustering is an unsupervised technique for data analysis whose objective is to group a
collection of objects into clusters based on similarity.
This is among the most popular techniques in data mining and can be mathematically described as follows.
Let $A = \{a^1,\ldots,a^n\}$ be a finite set of points in $\mathbb{R}^m$, which represent the data points to be
grouped. The goal is to partition $A$ into $k$ disjoint subsets $A^1,\ldots, A^k$, called clusters, such
that a clustering criterion is optimized.

There are many different criteria for the clustering problem. One of the most used is the \emph{Minimum Sum-of-Squares Clustering} criterion, where one tries to minimize the Euclidean distance of each data point to the centroid of its cluster~\cite{Bock1998,CYY2018,OB15}. Thus, each cluster $A_j$ is
identified by its center (or centroid) $x^j \in \mathbb{R}^m, j = 1,\ldots,k$. Letting $X:=\left(x^1,\ldots,x^k\right)\in\R^{m\times k}$, this gives rise to the following optimization problem:
$$
\text{minimize} \quad \varphi(X,\omega): = \frac{1}{n} \sum_{i=1}^{n} \sum_{j=1}^{k} \omega_{ij} \|x^j-a^i\|^2,
$$
where the binary variables $\omega_{ij}$ express the assignment of the point $a^i$
to the cluster $j$; i.e.,
$\omega_{ij} = 1$ if $a^i \in A^j$, and $\omega_{ij}=0$ otherwise. This problem can be equivalently reformulated as the following nonsmooth nonconvex unconstrained optimization problem (see \cite{CYY2018,OB15}):
\begin{equation}\label{eq:P_cluster}
\underset{X\in\mathbb{R}^{m\times k}}{\textrm{minimize}}\; \phi\!\left(X\right):=\frac{1}{n}\sum_{i=1}^{n}\min_{j=1,\ldots,k} \|x^j-a^i\|^2.
\end{equation}
As explained in~\cite{CYY2018,OB15}, we can write this problem as a DC problem of type~\eqref{eq:DC_2norm} by taking
\begin{gather*}
g\!\left(X\right):=\frac{1}{n}\sum_{i=1}^{n} \sum_{j=1}^{k} \left\|x^j-a^i\right\|^2+\frac{\rho}{2}\|X\|^2,\\
h\!\left(X\right):=\frac{1}{n}\sum_{i=1}^{n}\max_{j=1,\ldots,k} \sum_{t=1,t \not= j}^{k}\left\|x^t-a^i\right\|^2+\frac{\rho}{2}\|X\|^2,
\end{gather*}
for some $\rho\geq 0$, where $\|X\|$ is the Frobenius norm of $X$. Observe that both functions $g$ and $h$ are convex, and strongly convex if $\rho>0$.
Moreover, $g$ is differentiable, and the subdifferential of $h$ can be explicitly computed (see \cite[page 346]{OB15} or \cite[Equation~(3.21)]{CYY2018}).  
\begin{experiment}[Clustering the Spanish cities in the peninsula]\label{ex:Clustering_Spain}
Consider the problem of finding a partition into five clusters of the 4001 Spanish cities in the peninsula with more than 500 residents. For illustrating the difference between the iterations of DCA and BDCA, we present in \cref{fig:cluster_Spain} the result of applying 10 iterations of DCA and BDCA to the clustering problem~\eqref{eq:P_cluster} from a random starting point (composed by a quintet of points in $\R^2$), with the parameters $\rho=\frac{1}{10}$, $\alpha=0.1$, $\beta=0.5$ and $\overline{\lambda}_1=5$. Both algorithms converge to the same critical point, but it is apparent that the line search of BDCA makes it faster.
\begin{figure}[ht!]
\centering
\includegraphics[width=0.75\textwidth]{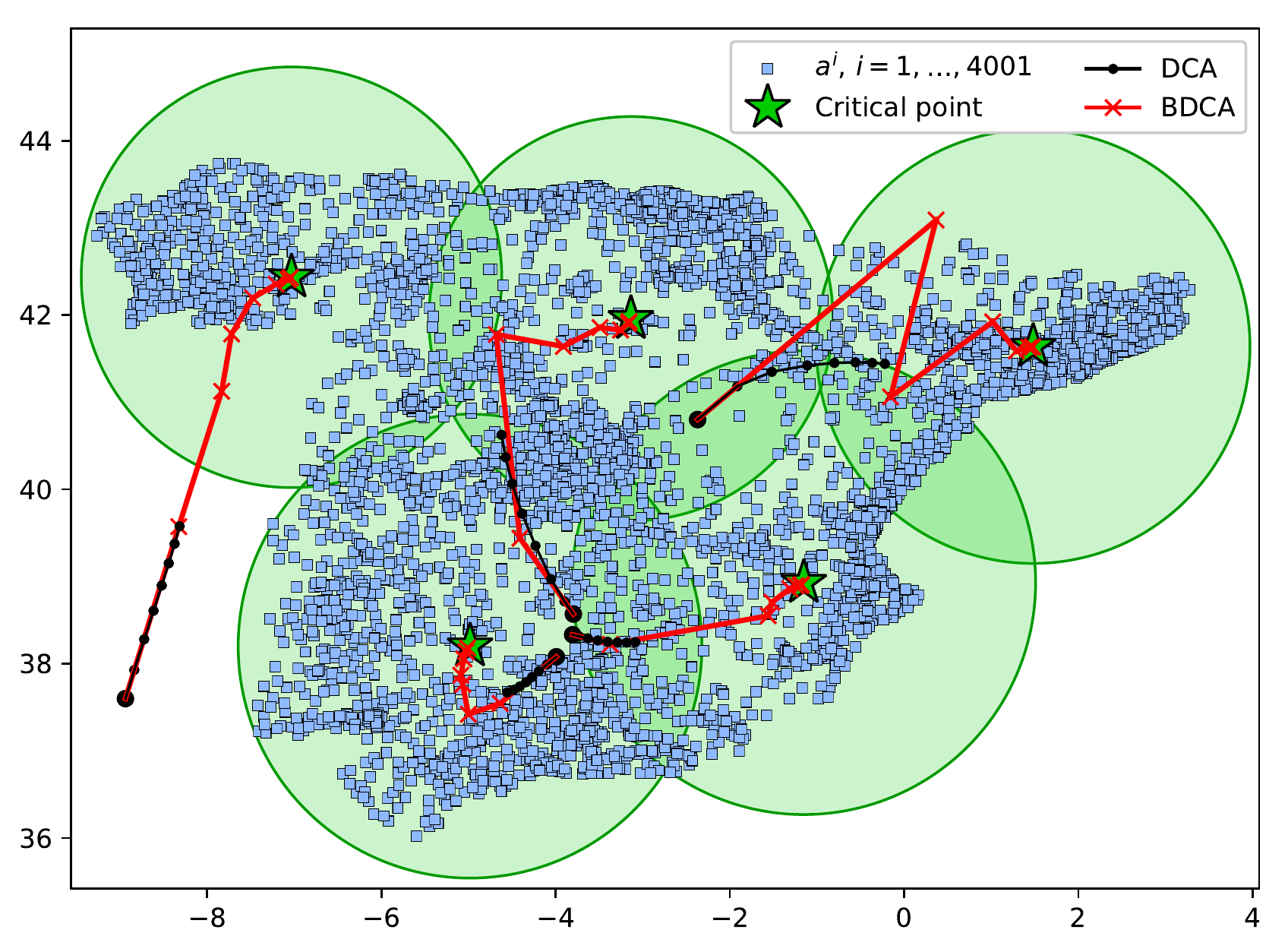}
\caption{Seven iterations of DCA and BDCA are computed from the same starting point for grouping the Spanish cities in the peninsula into five clusters.}\label{fig:cluster_Spain}
\end{figure}

Let us demonstrate that the behavior shown in \cref{fig:cluster_Spain} is not atypical. To do so, let us consider the same problem of the Spanish cities for a different number of clusters $k \in\{5,10,15,20,25,50,75,100\}$. For each of these values, we run BDCA for 100 random starting points with coordinates in ${]-9.26,  3.27[}\times{]36.02, 43.74[}$ (the range of the geographical coordinates of the cities). The algorithm was stopped when the relative error of the objective function $\phi$ was smaller than $10^{-3}$. Then, DCA was run from the same starting point until the same value of the objective function was reached, which did not happen in 31 instances because DCA  \emph{failed} (by which we mean that it converged to a worse critical point). In \cref{fig:cluster_Spain_ratio} we have plotted the ratios between the running time and the number of iterations, except for those instances where DCA failed. On average, BDCA was $16$ times faster than DCA, and DCA needed $18$ times more iterations to reach the same objective value as BDCA.
\begin{figure}[ht!]
\centering
\includegraphics[width=0.495\textwidth]{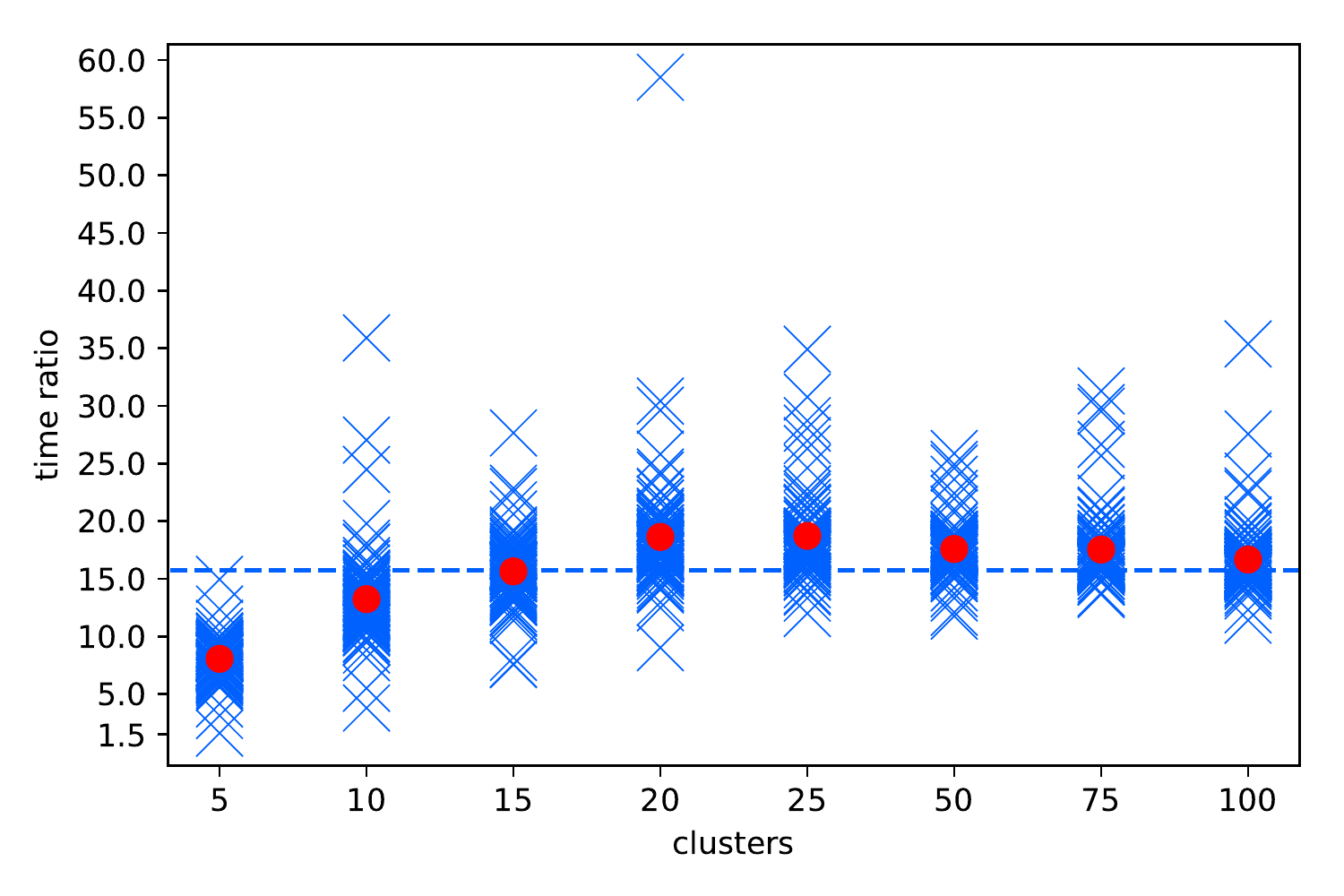}
\includegraphics[width=0.495\textwidth]{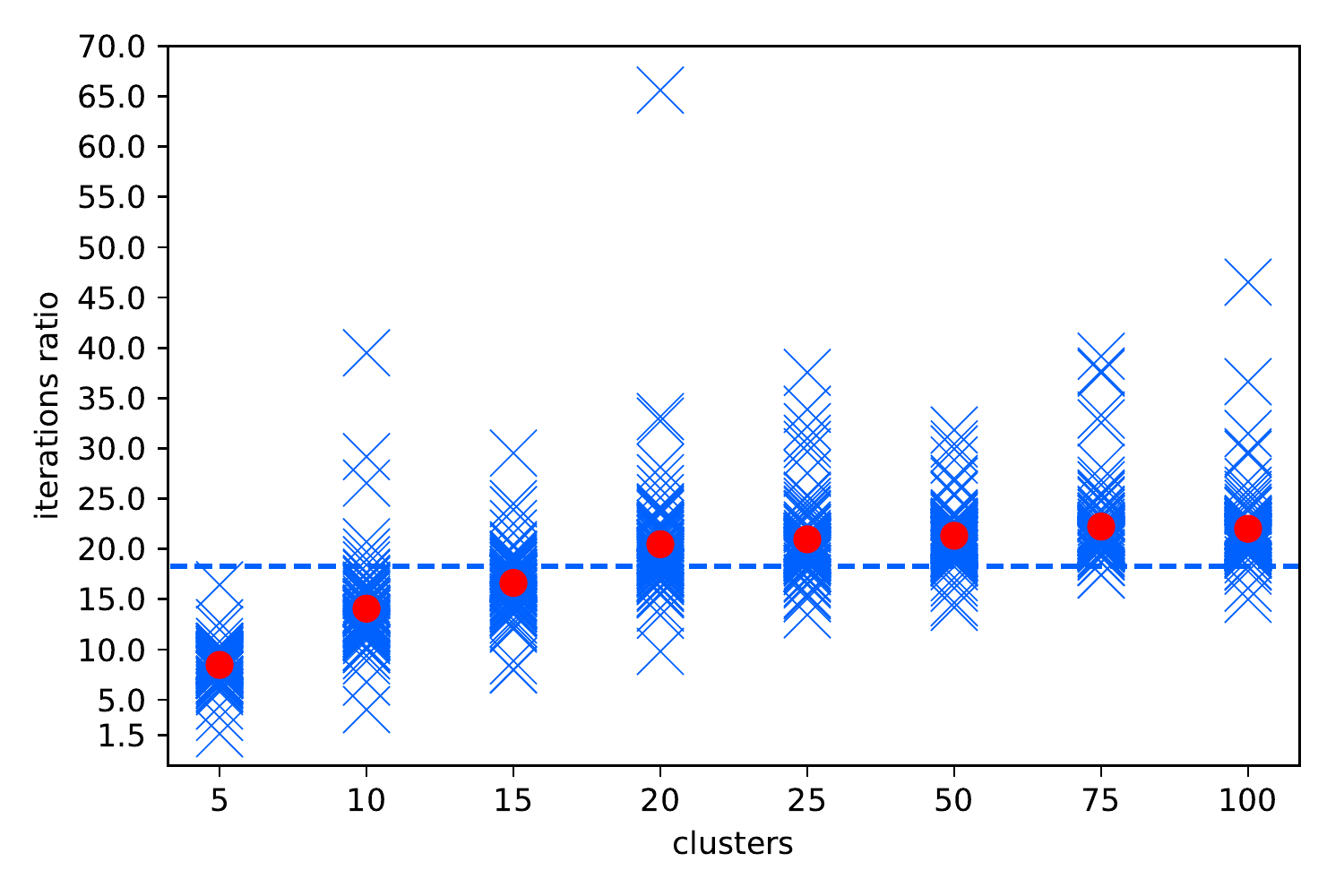}
\caption{Comparison between DCA and BDCA for solving the clustering problem of the cities in the Spanish peninsula described in \cref{ex:Clustering_Spain}.
We represent the ratios of running time (left) and number of iterations (right) between DCA and BDCA for 100 random instances for different values
of the number of clusters $k\in\{5,10,15,20,25,50,75,100\}$. The dashed line shows the overall average ratio, and the red dots represent the average
ratio for each value of $k$.}\label{fig:cluster_Spain_ratio}
\end{figure}
\end{experiment}

\begin{experiment}[Clustering random points in an $m$-dimensional box]\label{exp:2} In this numerical experiment, we generated $n$ random points in $\R^m$ whose coordinates were drawn from a normal distribution having a mean of $0$ and a standard deviation of $10$, with $n\in\{500,1000,5000,10,\!000\}$ and $m\in\{2,5,10,20\}$. For each pair of values of $n$ and~$m$, ten random starting points were chosen and BDCA was run to solve the $k$-clustering problem until the relative error of the objective function was smaller than $10^{-3}$, with $k \in\{5,10,15,20,25,50,75,100\}$. As in~\cref{ex:Clustering_Spain}, we run DCA from the same starting point than BDCA until the same value of the objective function was reached. The DCA failed to do so in 123 instances. The ratios between the respective running times are shown in \cref{fig:cluster_random}. On average, BDCA was $13.7$ times faster than DCA.
\begin{figure}[ht!]
\centering
\includegraphics[width=0.5\textwidth]{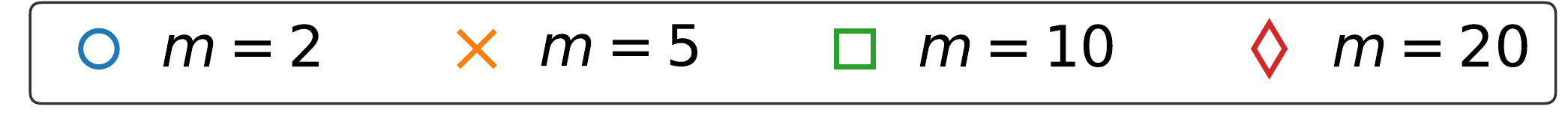}\\
\subfigure[$n=500$]{\includegraphics[width=0.495\textwidth]{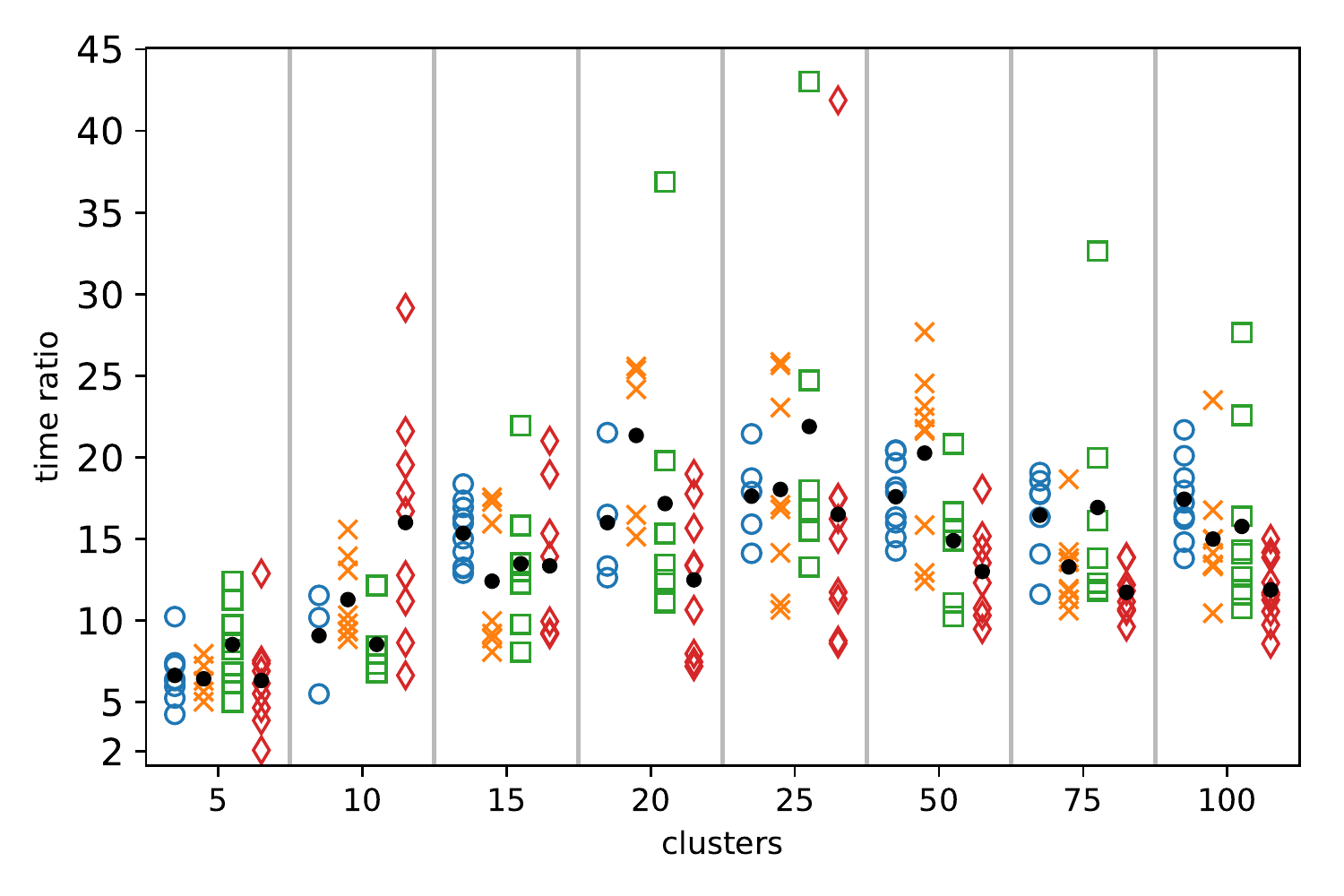}}
\subfigure[$n=1000$]{\includegraphics[width=0.495\textwidth]{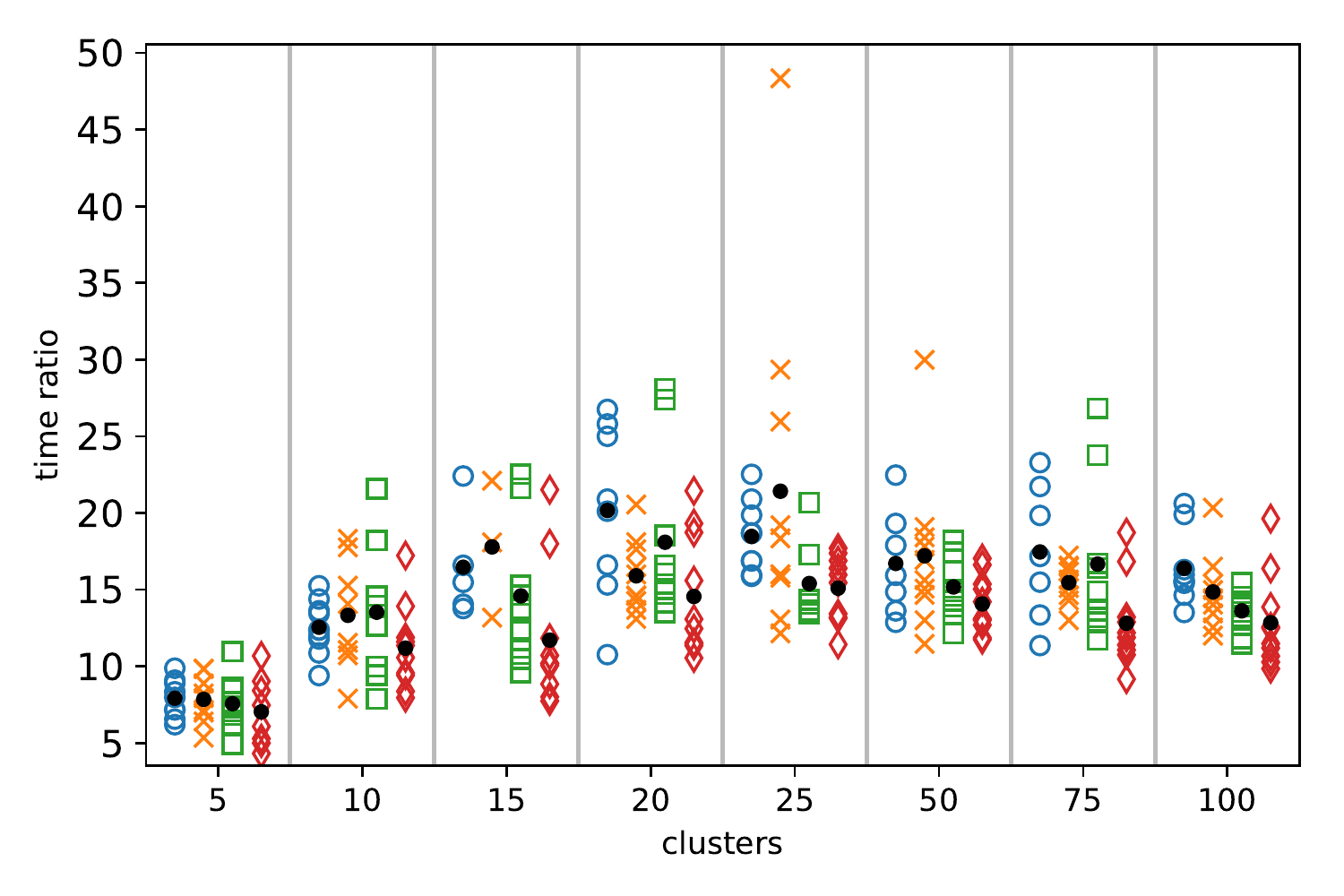}}
\subfigure[$n=5000$]{\includegraphics[width=0.495\textwidth]{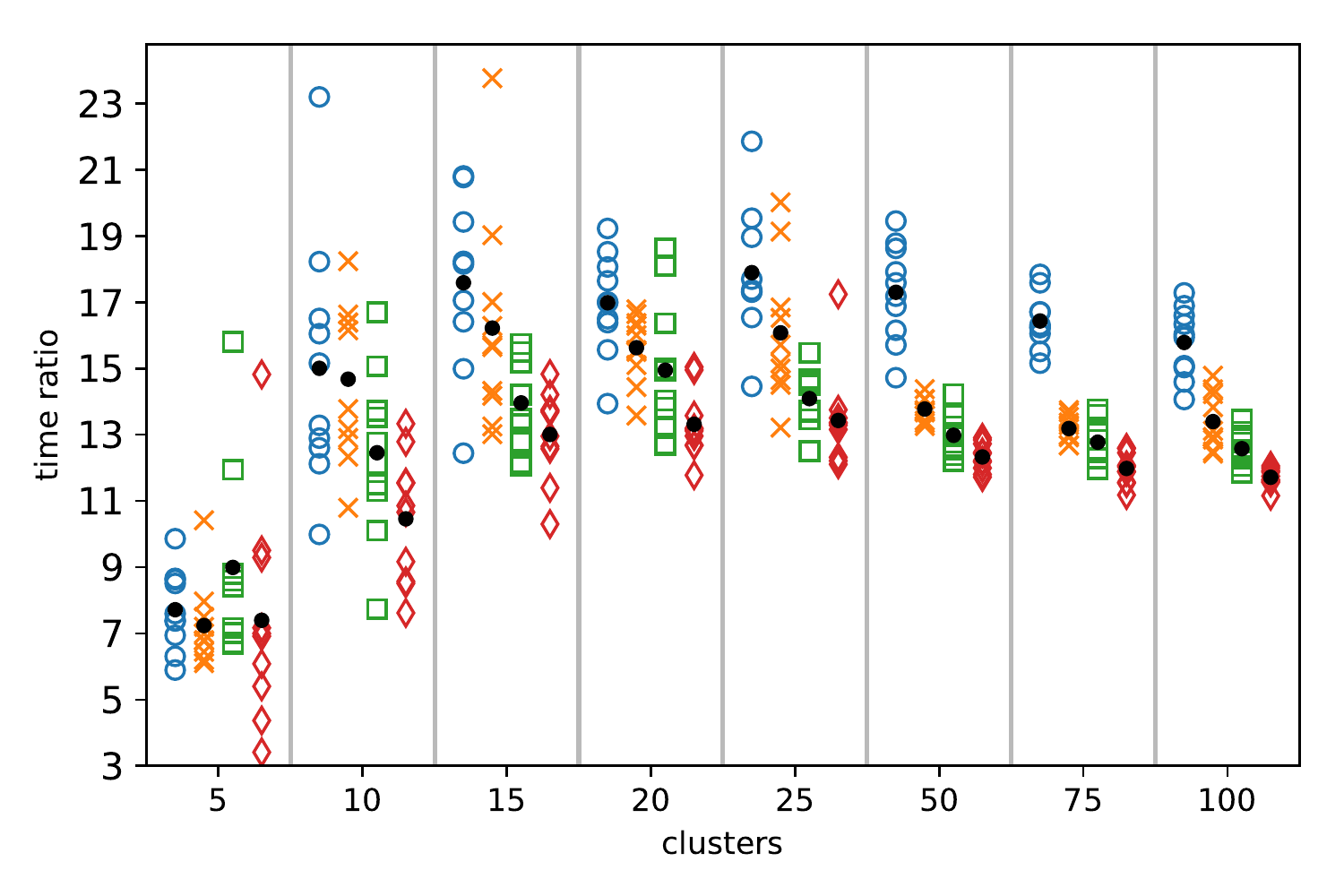}}
\subfigure[$n=10,\!000$]{\includegraphics[width=0.495\textwidth]{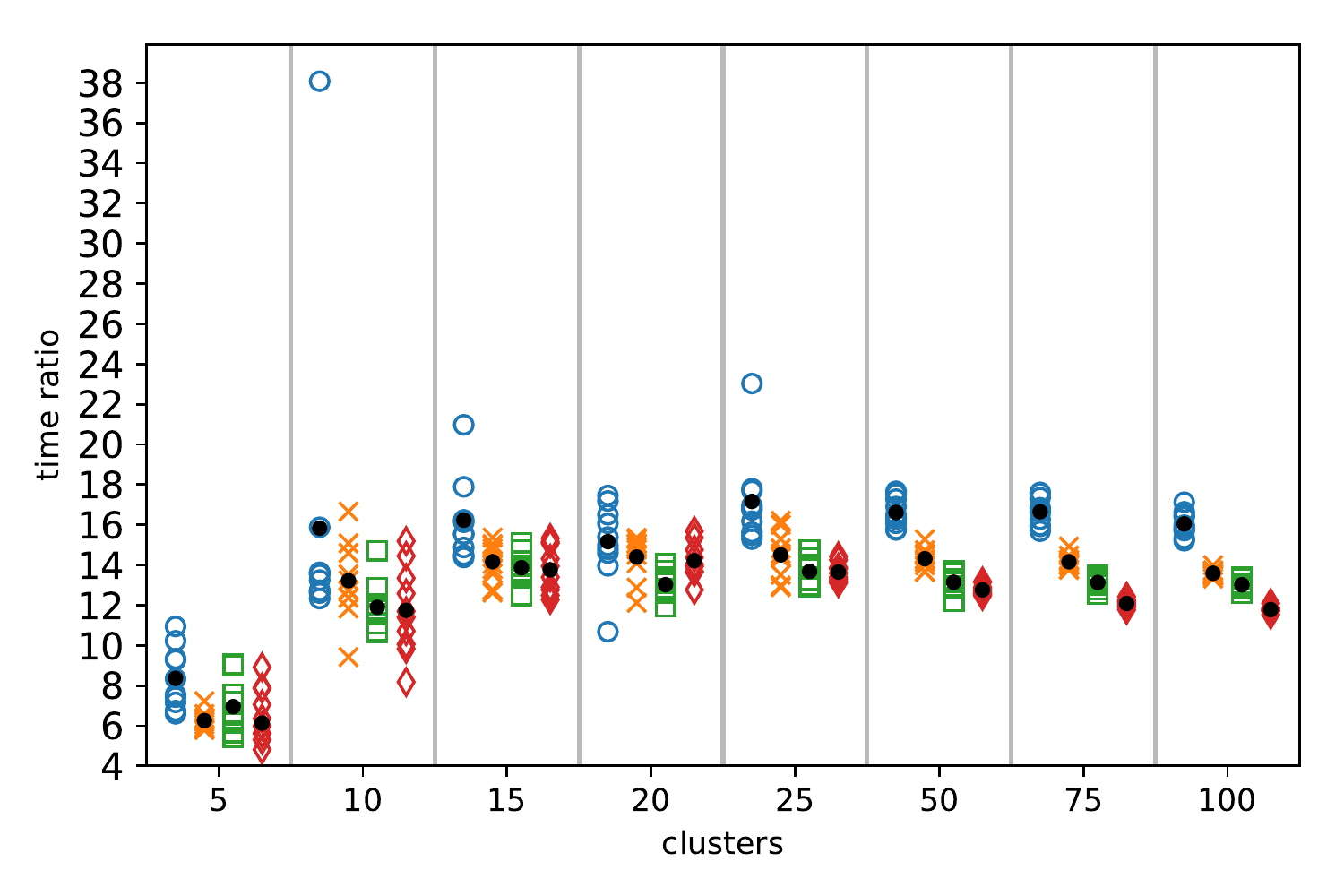}}
\caption{Comparison between DCA and BDCA for solving the clustering problems with random data described in \cref{exp:2}. For each value of $n\in\{500,1000,5000,10,\!000\}$ and $m\in\{2,5,10,20\}$ we represent the ratios of running time between DCA and BDCA for 10 random starting points for different values of the number of clusters $k\in\{5,10,15,20,25,50,75,100\}$. The black dots represent the average ratios.}\label{fig:cluster_random}
\end{figure}
\end{experiment}

\subsection {The Multidimensional Scaling Problem} \label{MDS}
Given only a table of distances between some objects, known as
the \emph{dissimilarity matrix}, \emph{Multidimensional Scaling (MDS)} is a technique that permits to represent the data in a small number of dimensions (usually two or three). If the objects are defined by $n$ points $x^1,x^2,\ldots,x^n$ in $\R^q$, the entries $\delta_{ij}$ of the dissimilarity matrix can be defined by the Euclidean distance between these points:
$$
\delta_{ij}=\|x^i-x^j\|:=\text{d}_{ij}(X),
$$
where we denote by $X$ the $n\times q$ matrix whose rows are $x^1,x^2,\ldots,x^n$.

Given a target dimension $p\leq q$, the metric MDS problem consists in finding $n$ points in $\mathbb{R}^p$, which are represented by an $n \times p$ matrix $X^*$, such that the quantity
$$
\mathit{Stress} (X^*): = \sum_{i<j} w_{ij} \left( \text{d}_{ij} (X^*)-\delta_{ij}\right)^2
$$
is smallest, where $w_{ij}$ are nonnegative weights. As shown in~\cite[p. 236]{Tao01}, this problem can be equivalently reformulated as a DC problem of type~\eqref{eq:DC_2norm} by setting
\begin{gather*}
g(X):=\frac{1}{2} \sum_{i<j} w_{ij} \text{d}^2_{ij} (X)+\frac{\rho}{2}\|X\|^2,\\
h(X):= \sum_{i<j} w_{ij} \delta_{ij} \text{d}_{ij} (X)+\frac{\rho}{2}\|X\|^2,
\end{gather*}
for some $\rho\geq 0$. Moreover, it is clear that $g$ is differentiable while $h$ is not. However, the subgradient of $h$ can be explicitly computed, see~\cite[Section 4.2]{Tao01}. Both functions are strongly convex for any $\rho>0$.

For this problem we replicated some of the numerical experiments in~\cite{Tao01}, where the authors demonstrate the good performance of DCA for solving MDS problems. Our main aim here is showing that, even for those problems where DCA works well in practice, BDCA is able to outperform it.

In our experiments, we set the weights $w_{ij}=1$ and the starting points were generated as in~\cite{Tao01}. First, we randomly chose a matrix $\widetilde{X}_0\in \mathbb{R}^{n\times p}$ with entries in ${]0,10[}$. Then, the starting point was set as $X_0:=\left(I-(1/n)ee^T\right)\widetilde{X}_0$, where $I$ and $e$ denote the identity matrix and the vector of ones in $\R^n$, respectively.  We used the parameters $\rho=\frac{1}{np}$, $\alpha=0.05$, $\overline{\lambda}_1=3$ and $\beta=0.1$.


\newlength{\tam}
\setlength{\tam}{10.2em}
\begin{experiment}[MDS for Spanish cities]\label{ex:Spain}
Consider the dissimilarity matrix defined by the distances between the 4155 Spanish cities with more than 500 residents, including this time those outside the peninsula to make the problem more difficult. 
The optimal value of this MDS problem is zero. In \cref{fig:Spain}(b) we have represented a starting point of the type $X_0:=\left(I-(1/4155)ee^T\right)\widetilde{X}_0$, where $\widetilde{X}_0\in \mathbb{R}^{4155\times 2}$ was randomly chosen with entries in ${]0,10[}$. In \cref{fig:Spain}(c)-(k) we plot the iterations of DCA and BDCA. As shown in \cref{fig:Spain}(a), despite both DCA and BDCA converged to the optimal solution, DCA required five times more iterations than BDCA to reach the same accuracy.

\begin{figure}[ht!]\centering
\subfigure[Value of the objective function]{\includegraphics[height=1.13\tam]{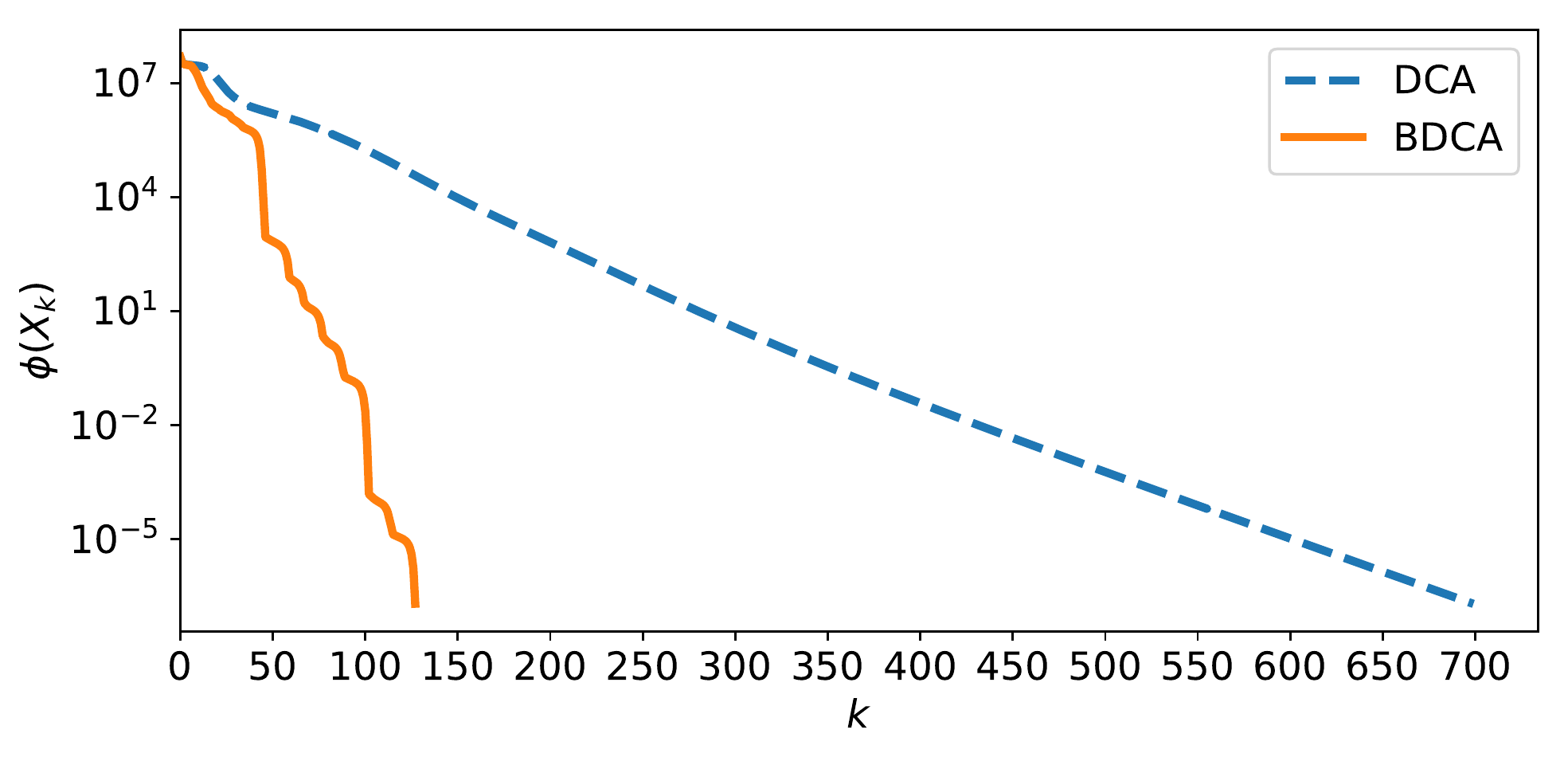}}
\subfigure[Starting point]{\includegraphics[height=1.13\tam]{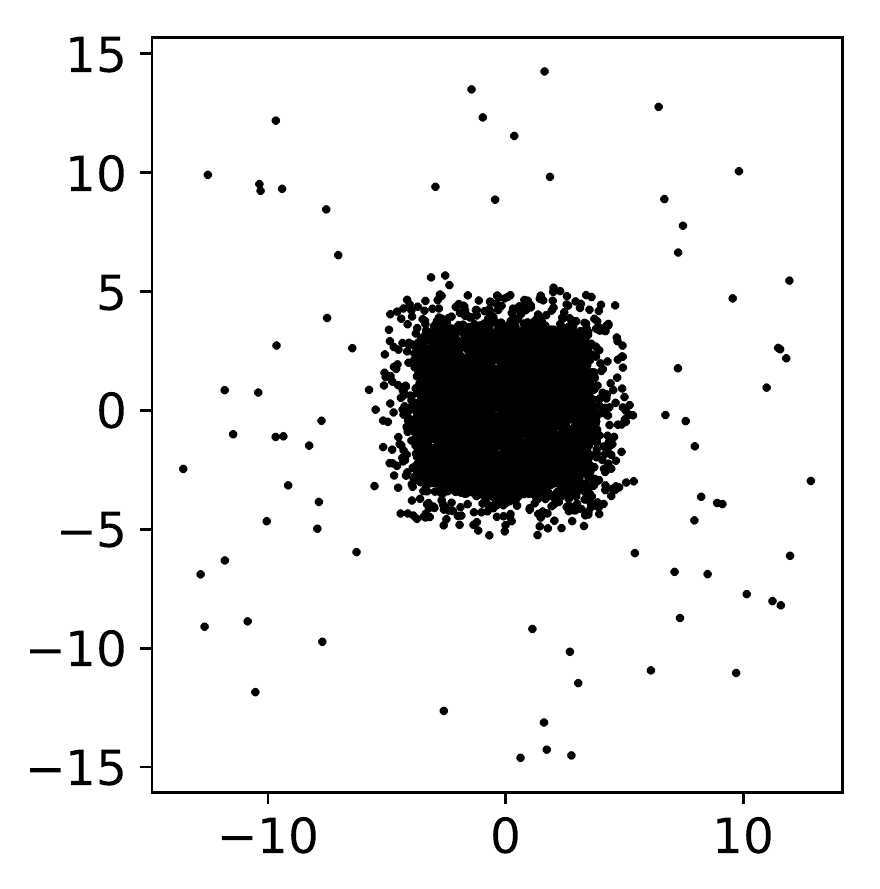}}
\subfigure[25 iterations of BDCA]{\includegraphics[height=1\tam]{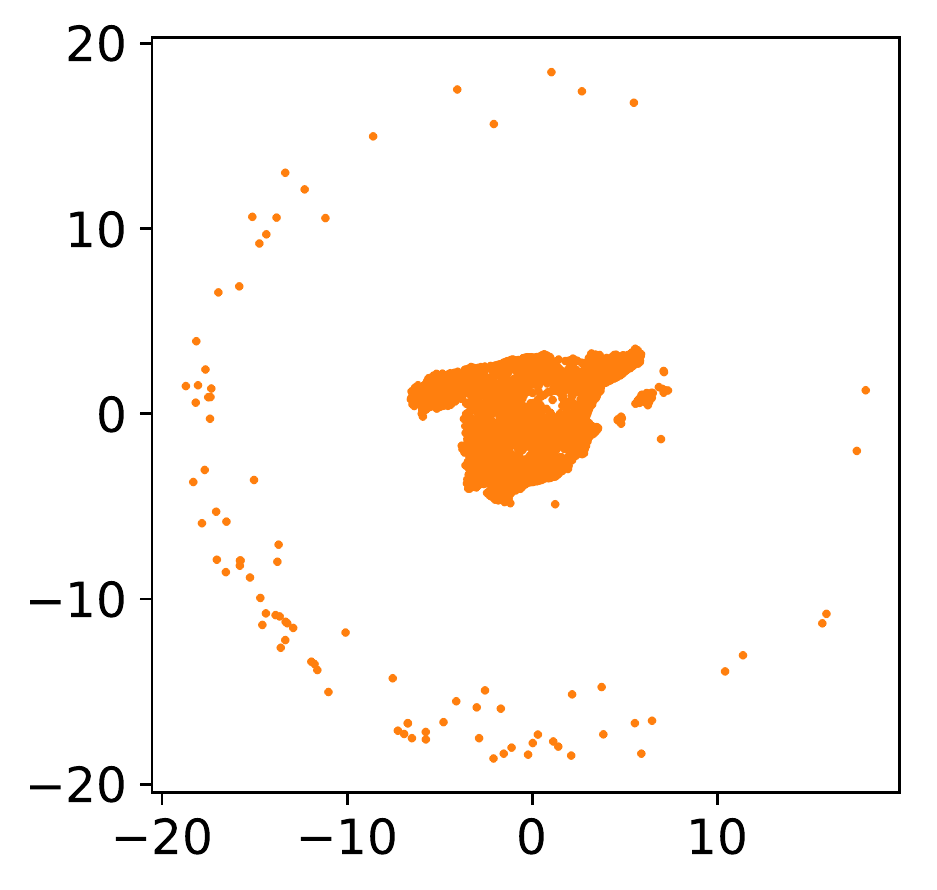}}
\subfigure[50 iterations of BDCA]{\includegraphics[height=1\tam]{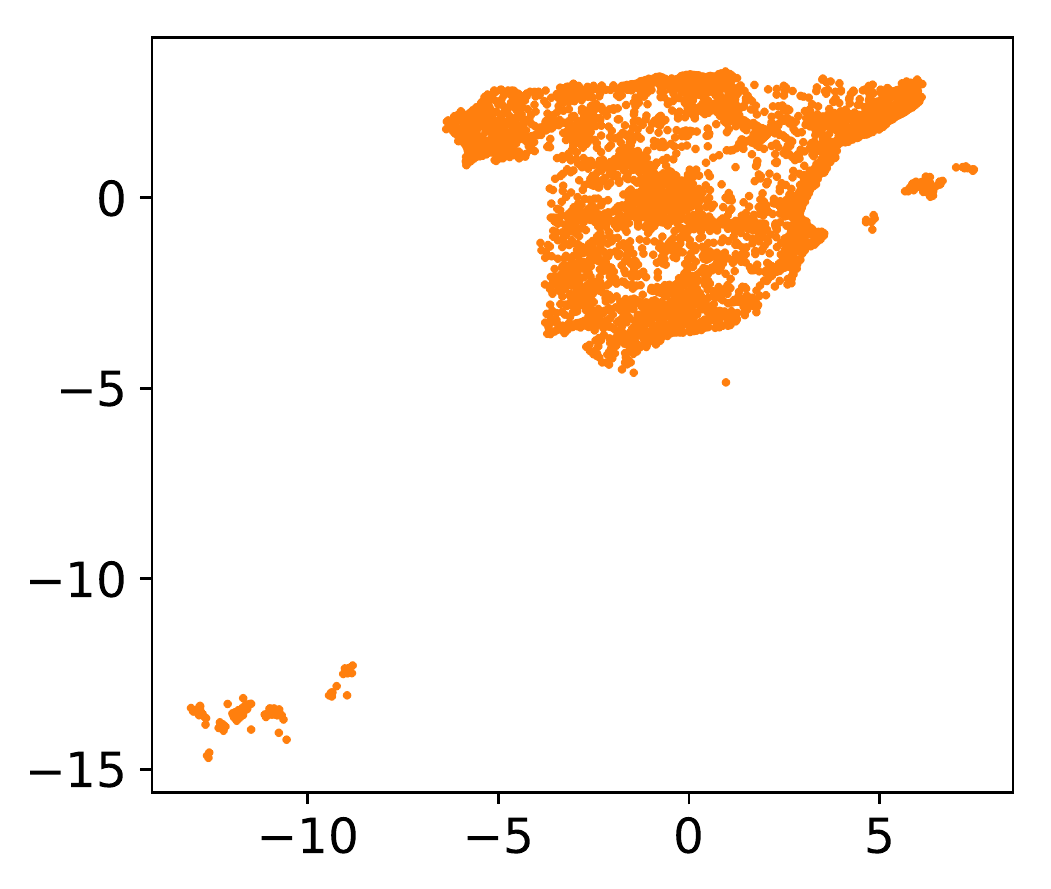}}
\subfigure[100 iterations of BDCA]{\includegraphics[height=1\tam]{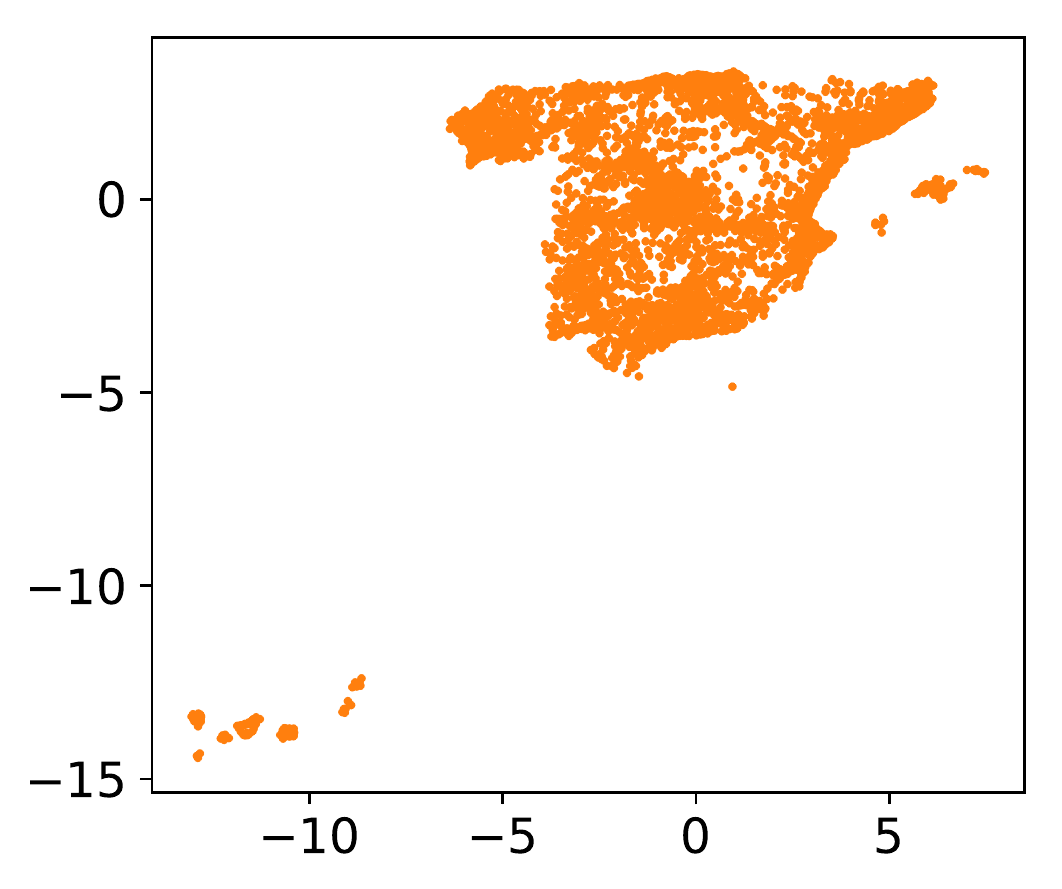}}\\
\subfigure[25 iterations of DCA]{\includegraphics[height=1.031\tam]{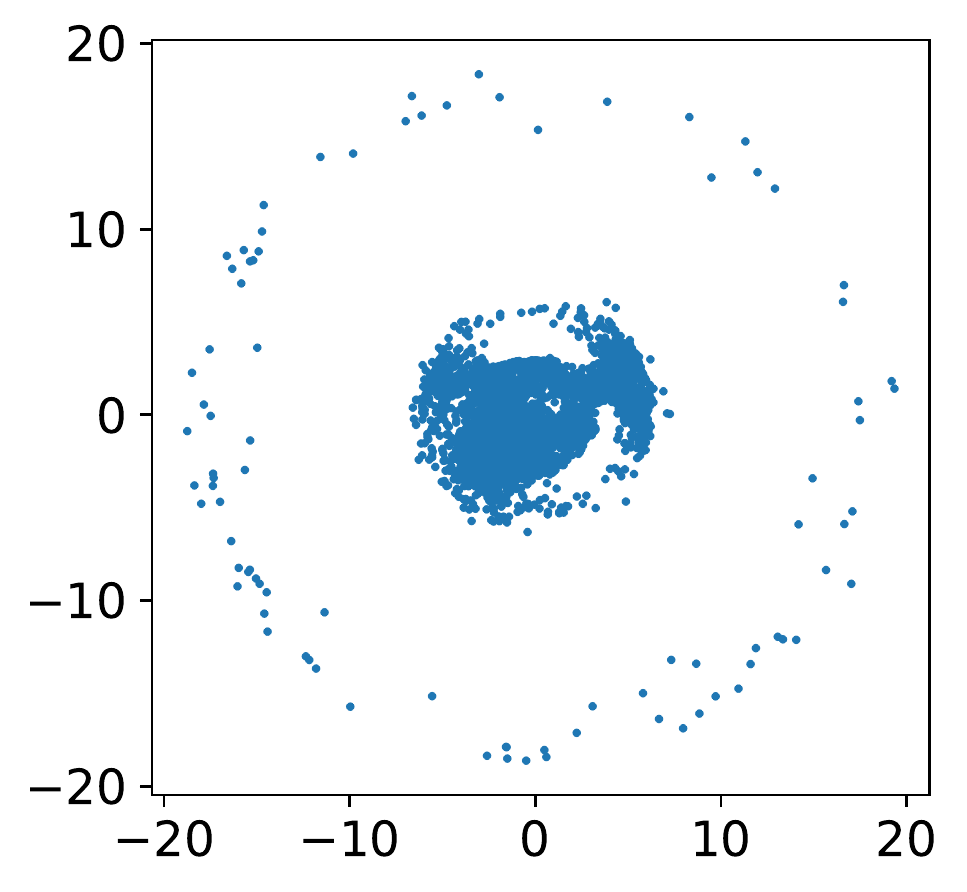}}
\subfigure[50 iterations of DCA]{\includegraphics[height=1.031\tam]{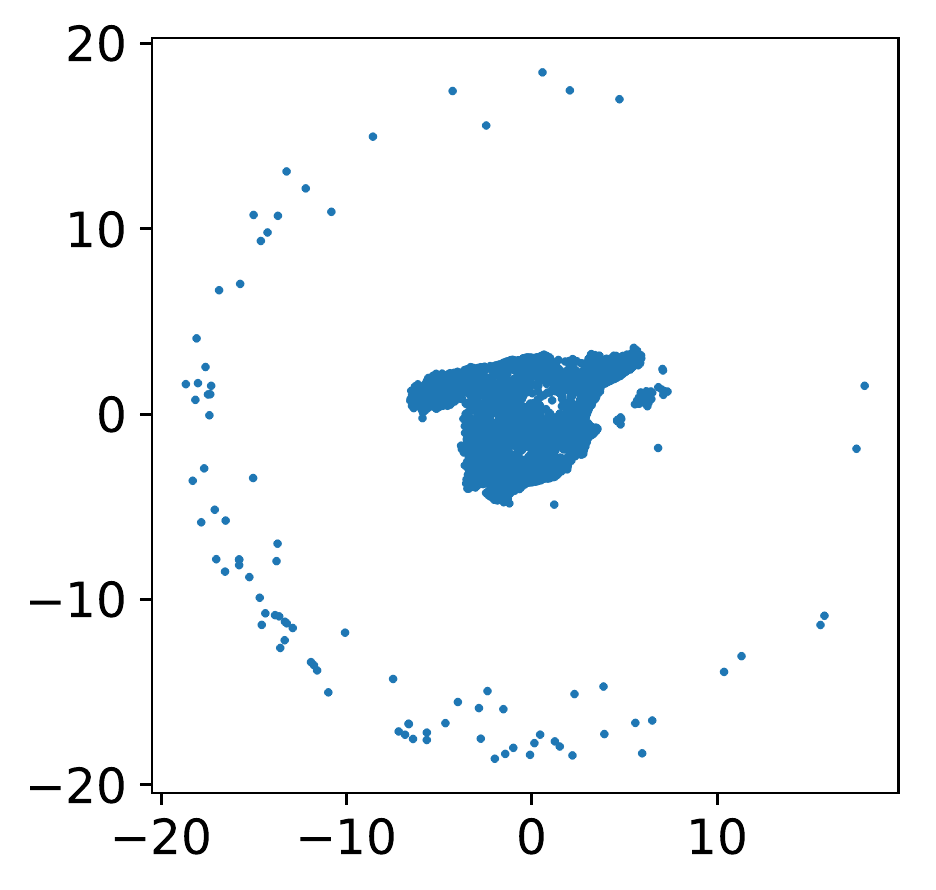}}
\subfigure[100 iterations of DCA]{\includegraphics[height=1.031\tam]{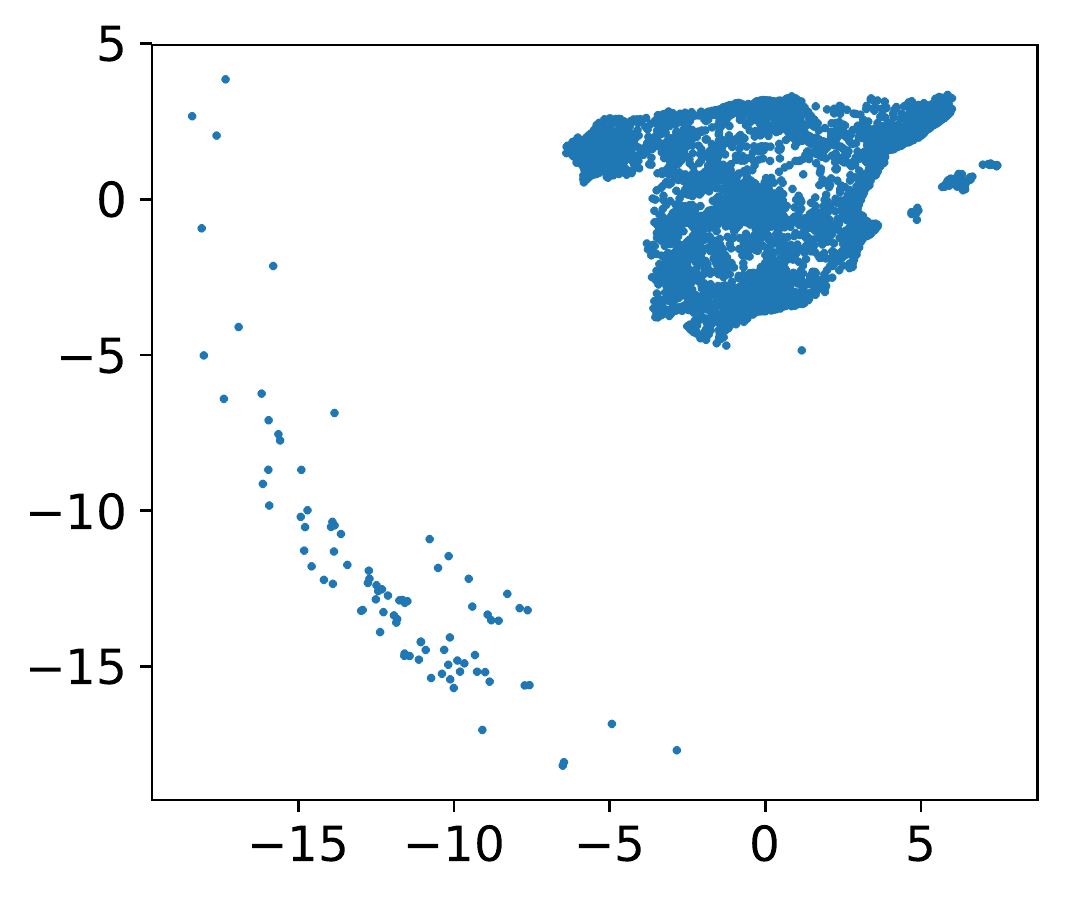}}
\subfigure[150 iterations of DCA]{\includegraphics[height=0.952\tam]{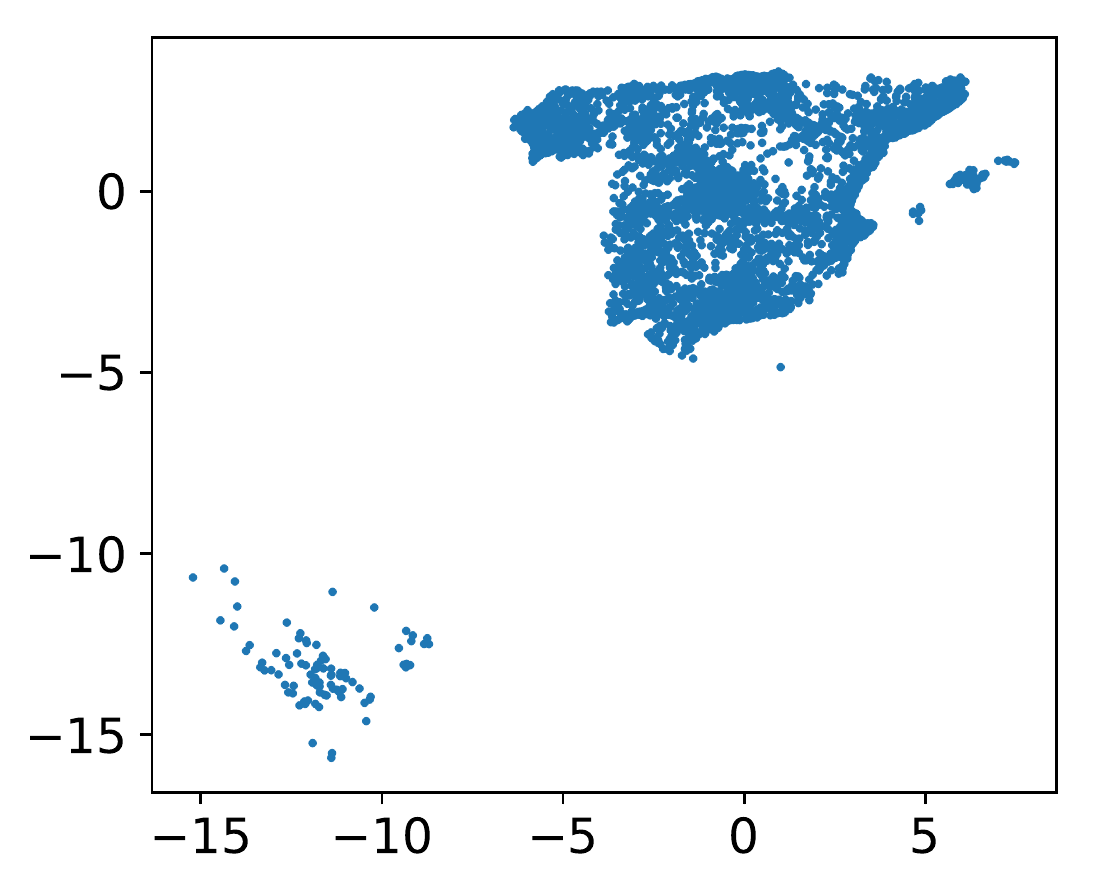}}
\subfigure[200 iterations of DCA]{\includegraphics[height=0.952\tam]{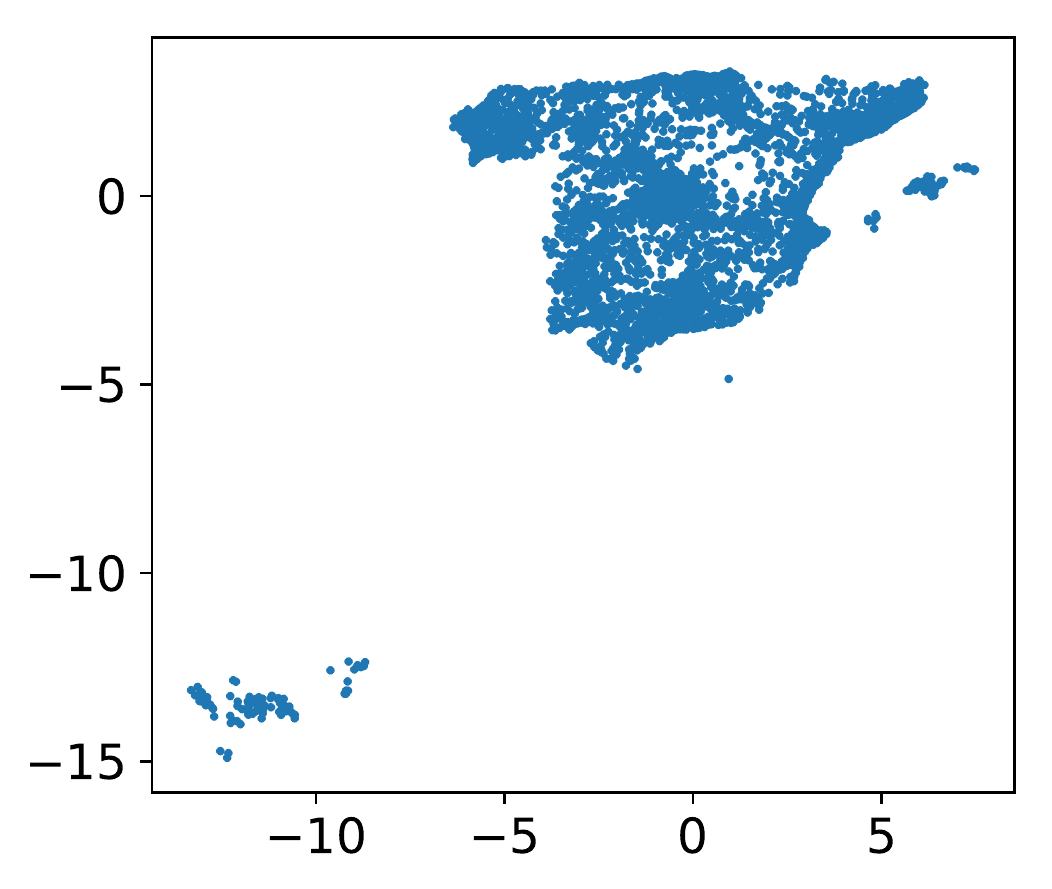}}
\subfigure[400 iterations of DCA]{\includegraphics[height=0.952\tam]{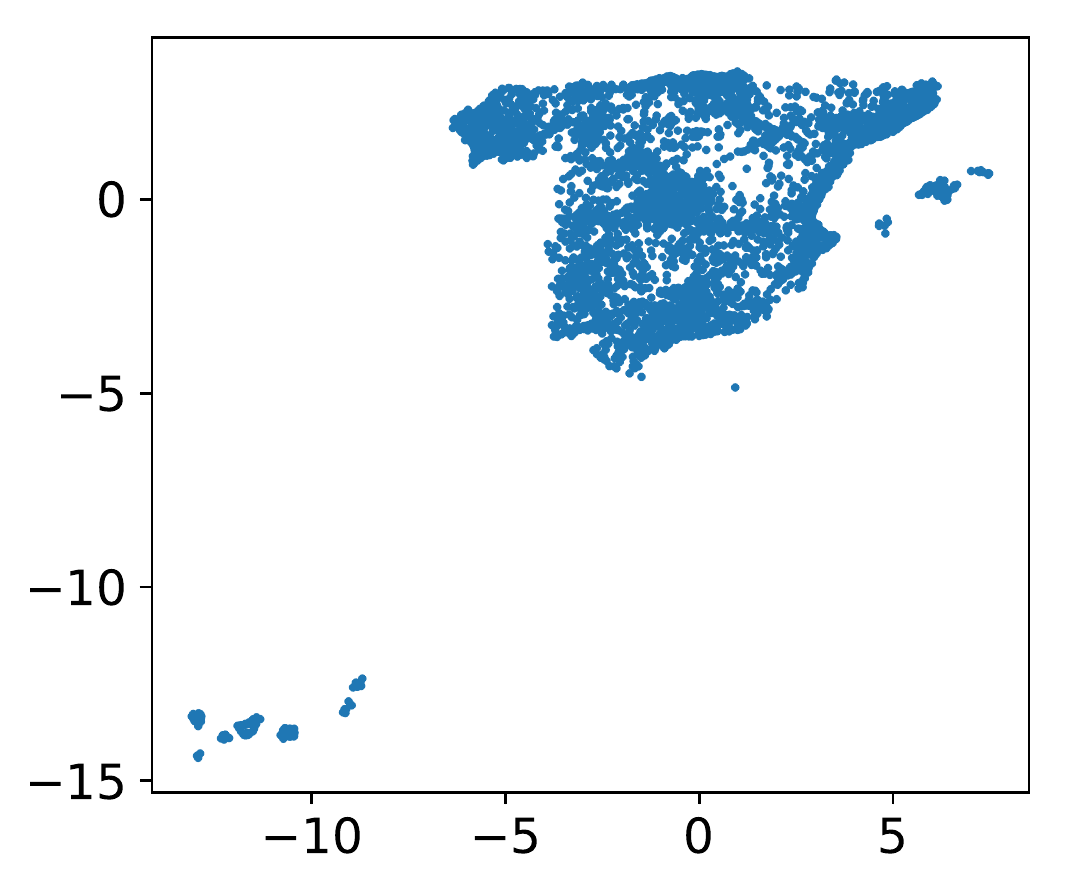}}
\caption{Comparison between DCA and BDCA when they are applied to the MDS problem of the Spanish cities described in \cref{ex:Spain} from the same random starting point.\label{fig:Spain}}
\end{figure}
\begin{figure}[ht!]
\centering
\includegraphics[width=0.495\textwidth]{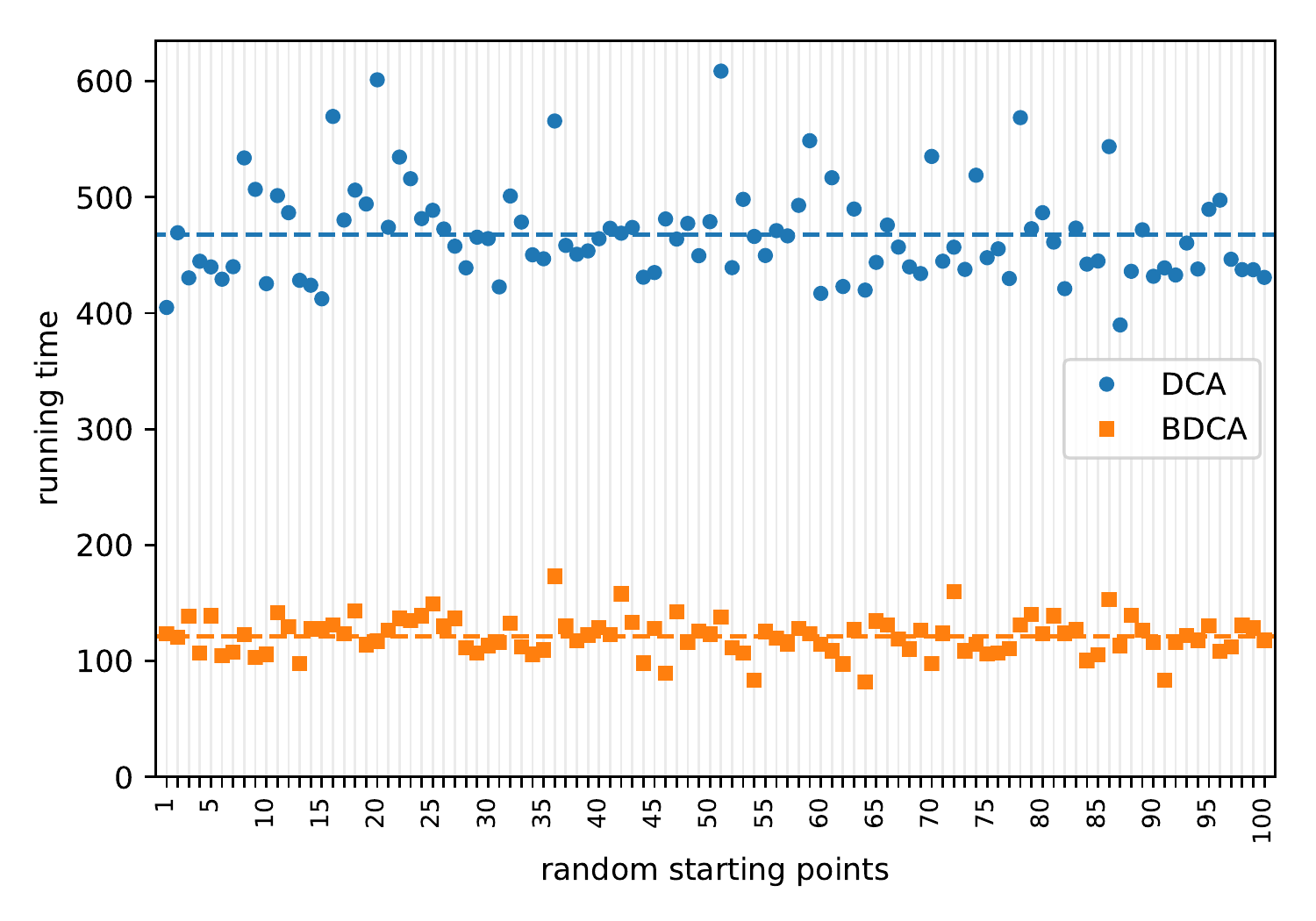}
\includegraphics[width=0.495\textwidth]{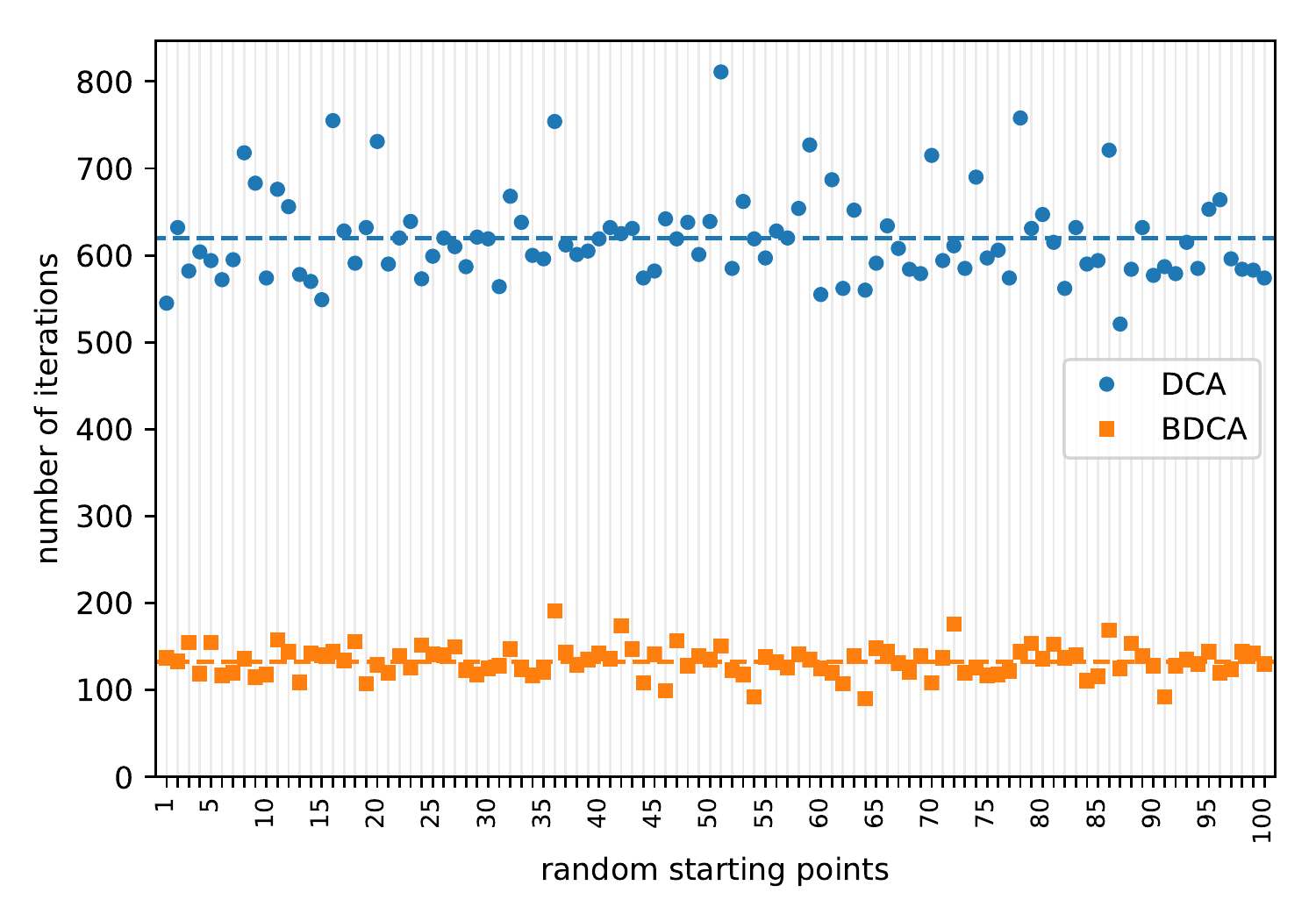}
\caption{Comparison between DCA and BDCA for solving the MDS problem of the Spanish cities described in \cref{ex:Spain}.
We represent the running time (left) and number of iterations (right) of DCA and BDCA for 100 random instances.
The dashed lines show the averages.}\label{fig:MDS_Spain_test}
\end{figure}
\end{experiment}

To demonstrate that the advantage shown in \cref{fig:Spain} is not unusual, we run both algorithms from 100 different random starting points until either the value of the objective function was smaller than $10^{-6}$, or until $\phi(X_k)-\phi(X_{k+1})<10^{-6}$. This second stopping criterion was used in 32 instances, and in all of them the value of $\phi$ was approximately equal to $26683.66$. The running time and the number of iterations of both algorithms is plotted in \cref{fig:MDS_Spain_test}. On average, BDCA was $3.9$ times faster than DCA. Further, BDCA was always more than $2.9$ times faster than DCA, and the number of iterations required by DCA was always more than $3.5$ times higher (on average, it was $4.7$ times higher). In fact, the minimum time required by DCA within all the random instances (389.9 seconds) was $2.2$ times higher than the maximum time spent by BDCA (173.2 seconds).

\begin{experiment}[MDS with random data]\label{exp:4}
To test randomly generated data, we considered two cases:
\begin{itemize}
\item \textbf{Case~1}: the dissimilarities are distances between objects in $\R^p$; thus, the optimal value is $0$.
\item \textbf{Case~2}: the dissimilarities are distances between objects in $\R^{2p}$; hence, the optimal value is unknown a priori.
\end{itemize}
The data was obtained by generating a matrix $M$ in $\mathbb{R}^{n\times p}$ and $\mathbb{R}^{n\times 2p}$ with entries randomly drawn from a normal distribution having a mean of $0$ and a standard deviation of $10$. Then, the values of $\delta_{ij}$ were determined by the distance matrix between the rows of $M$. We used the same stopping criteria as in~\cite{Tao01}: for Case~1, the algorithms were stopped when the value of the merit function was smaller than $10^{-6}$, while for Case~2, they were stopped when the relative error of the objective function was smaller than~$10^{-3}$.

The ratios between the respective running times and number of iterations of DCA and BDCA are shown in \cref{fig:exp4_1}. On average, BDCA was $2.6$ times faster than DCA, and the advantage was bigger both for Case~1 and for $p=3$. For Case~2 we can find some instances where BDCA was only 1.5 times faster than DCA. In \cref{fig:exp4_2} we observe that these instances seem to be outliers, for which DCA was faster than usual. The value of the objective function with respect to time of both algorithms for a particular large random instance is plotted in \cref{fig:exp4_3}.

\begin{figure}[ht!]\centering
\subfigure[Case 1 (running time ratio)]{\includegraphics[width=0.48\textwidth]{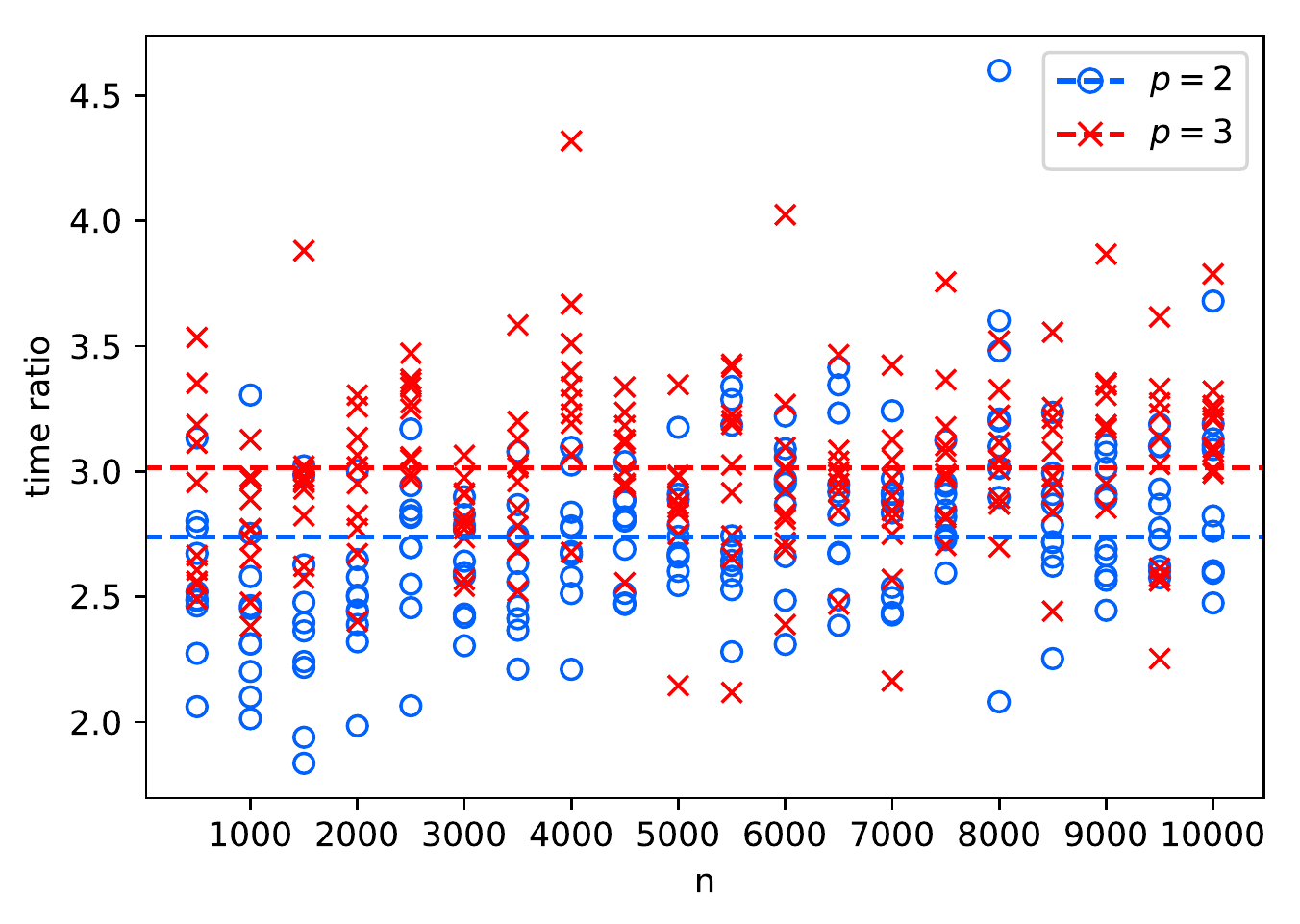}}\hfill
\subfigure[Case 1 (number of iterations ratio)]{\includegraphics[width=0.48\textwidth]{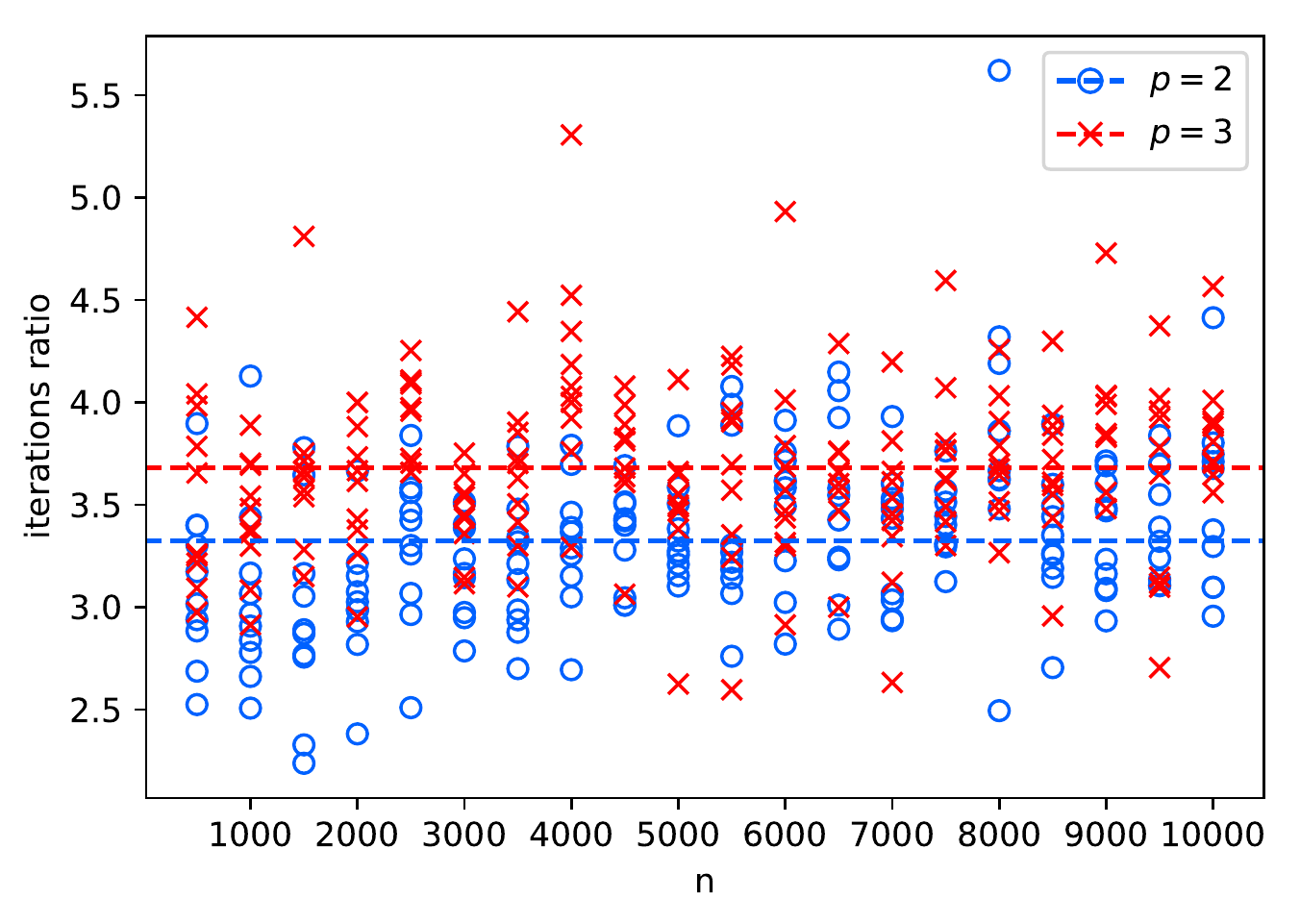}}
\subfigure[Case 2 (running time ratio)]{\includegraphics[width=0.48\textwidth]{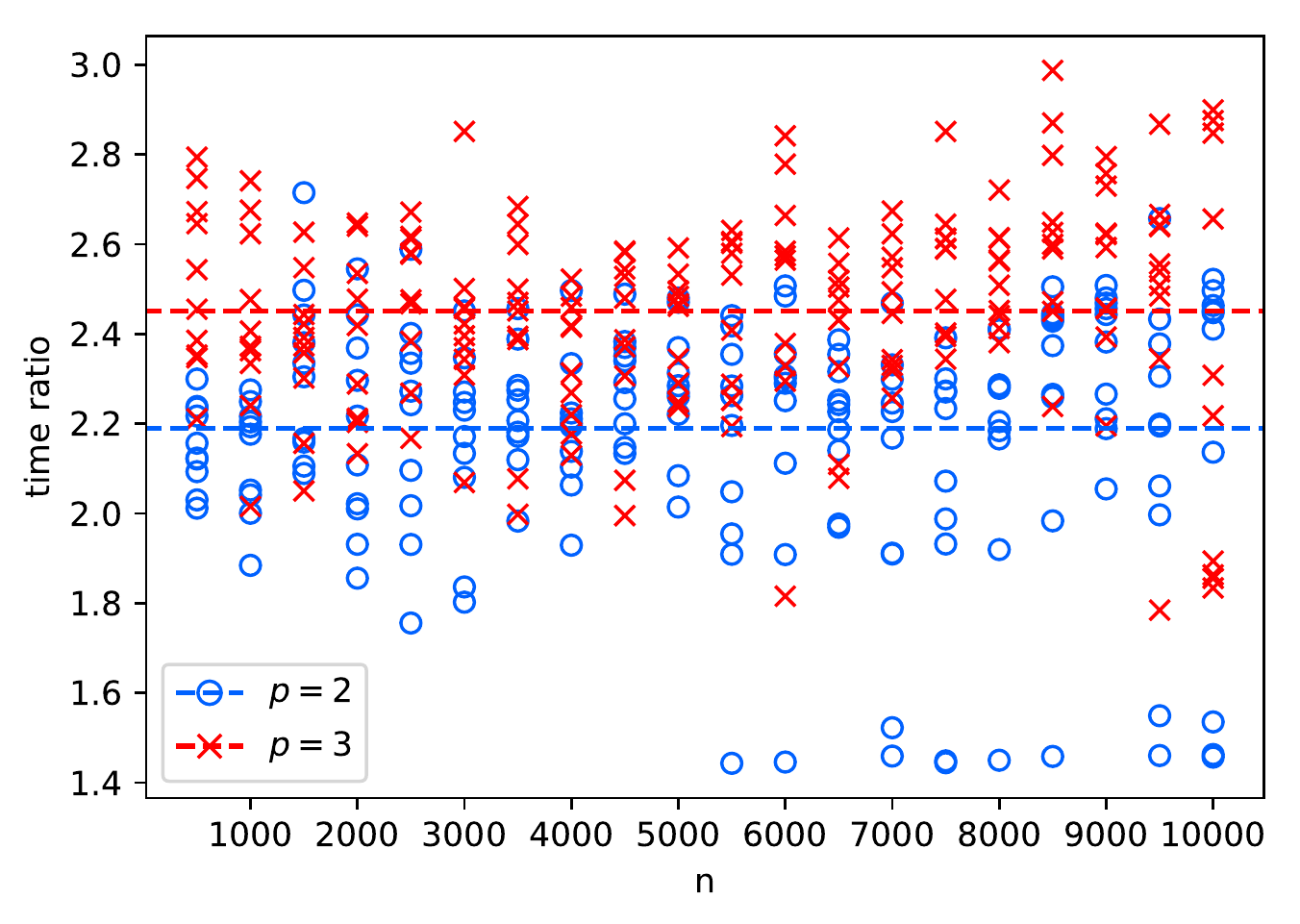}}\hfill
\subfigure[Case 2 (number of iterations ratio)]{\includegraphics[width=0.48\textwidth]{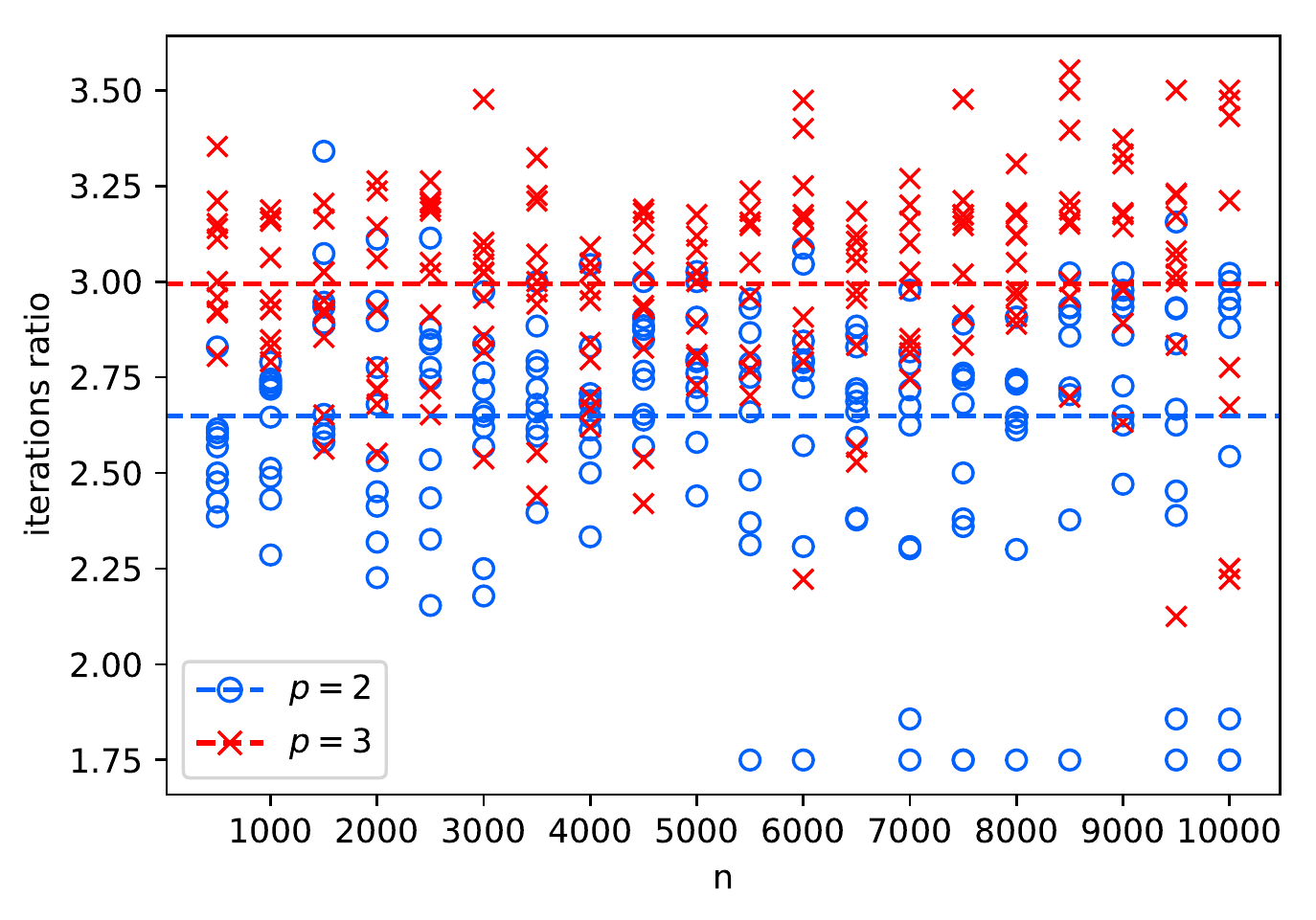}}
\caption{Comparison between DCA and BDCA for solving the MDS problems with random data described in \cref{exp:4}. We represent the ratios of running time and number of iterations between DCA and BDCA for ten random instances for each value of $n\in\{500,1000,\ldots,10,\!000\}$ and $p\in\{2,3\}$. For each $p$, the average value is represented with a dashed line.\label{fig:exp4_1}}
\end{figure}

\begin{figure}[ht!]\centering
\subfigure[Running time]{\includegraphics[width=0.48\textwidth]{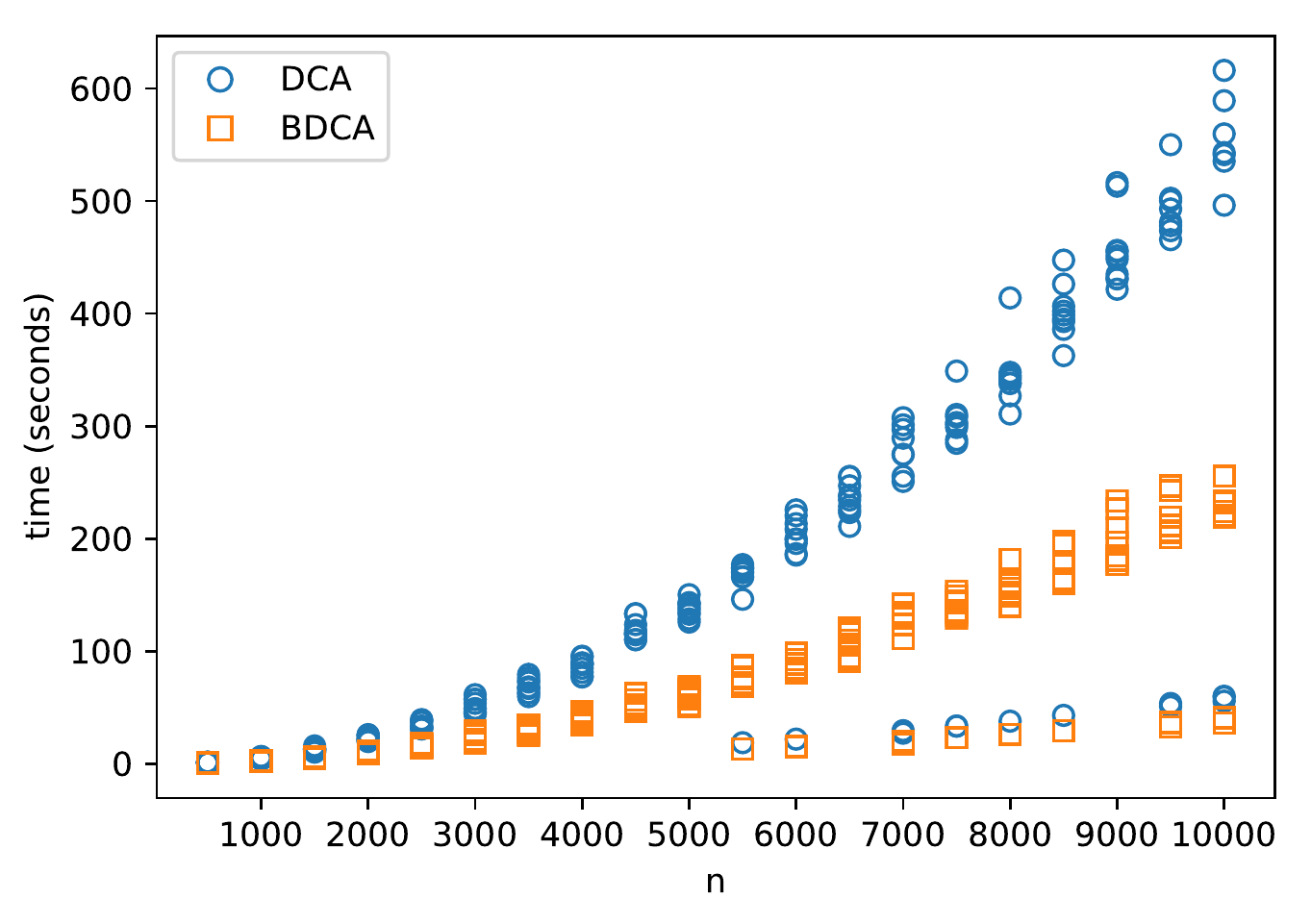}}\hfill
\subfigure[Number of iterations]{\includegraphics[width=0.48\textwidth]{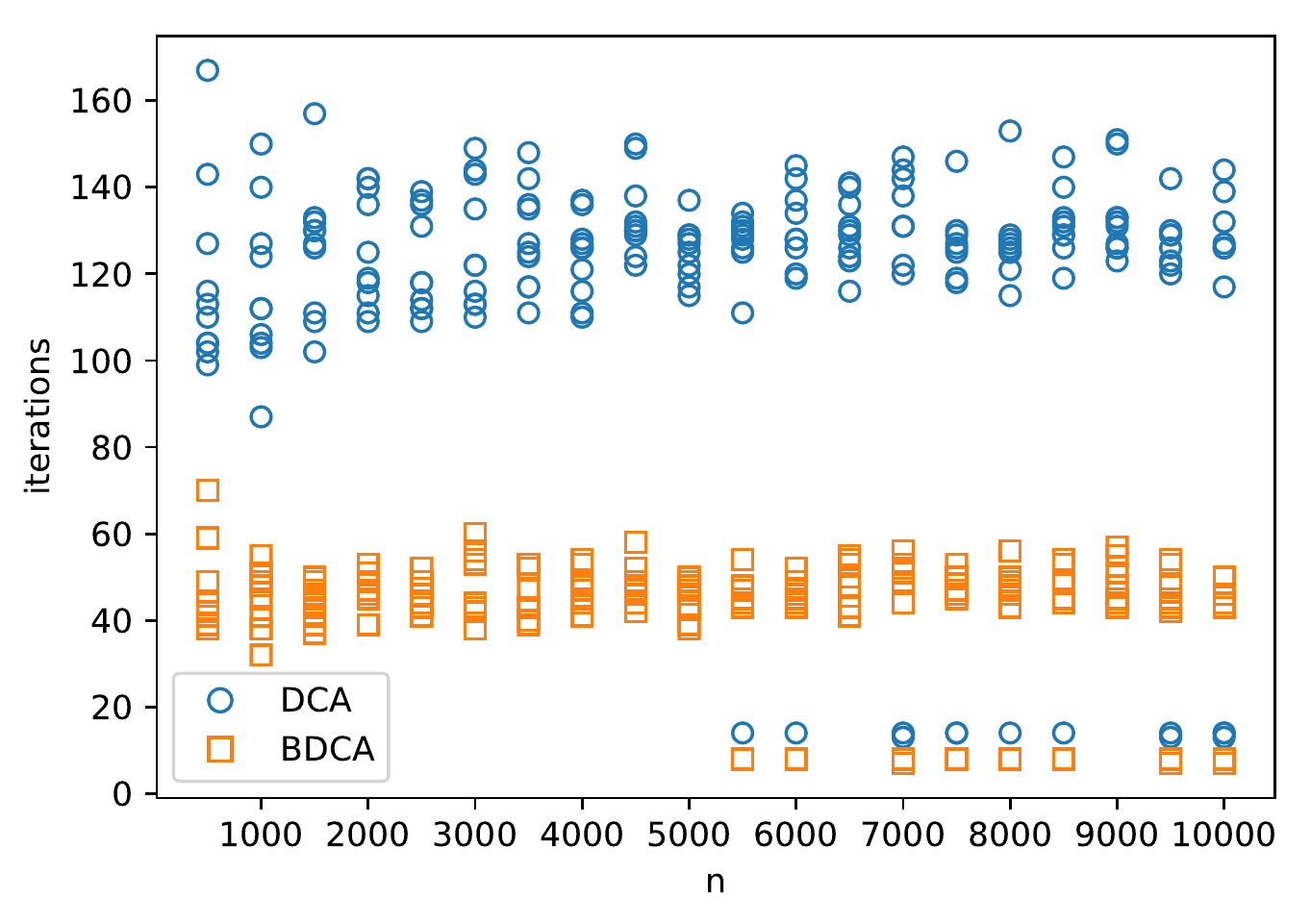}}
\caption{Running time and number of iterations for DCA and BDCA when applied to the random data described in \cref{exp:4} for Case~2 with  $p=2$.\label{fig:exp4_2}}
\end{figure}

\begin{figure}[ht!]\centering
\subfigure[Case 1]{\includegraphics[height=0.32\textwidth]{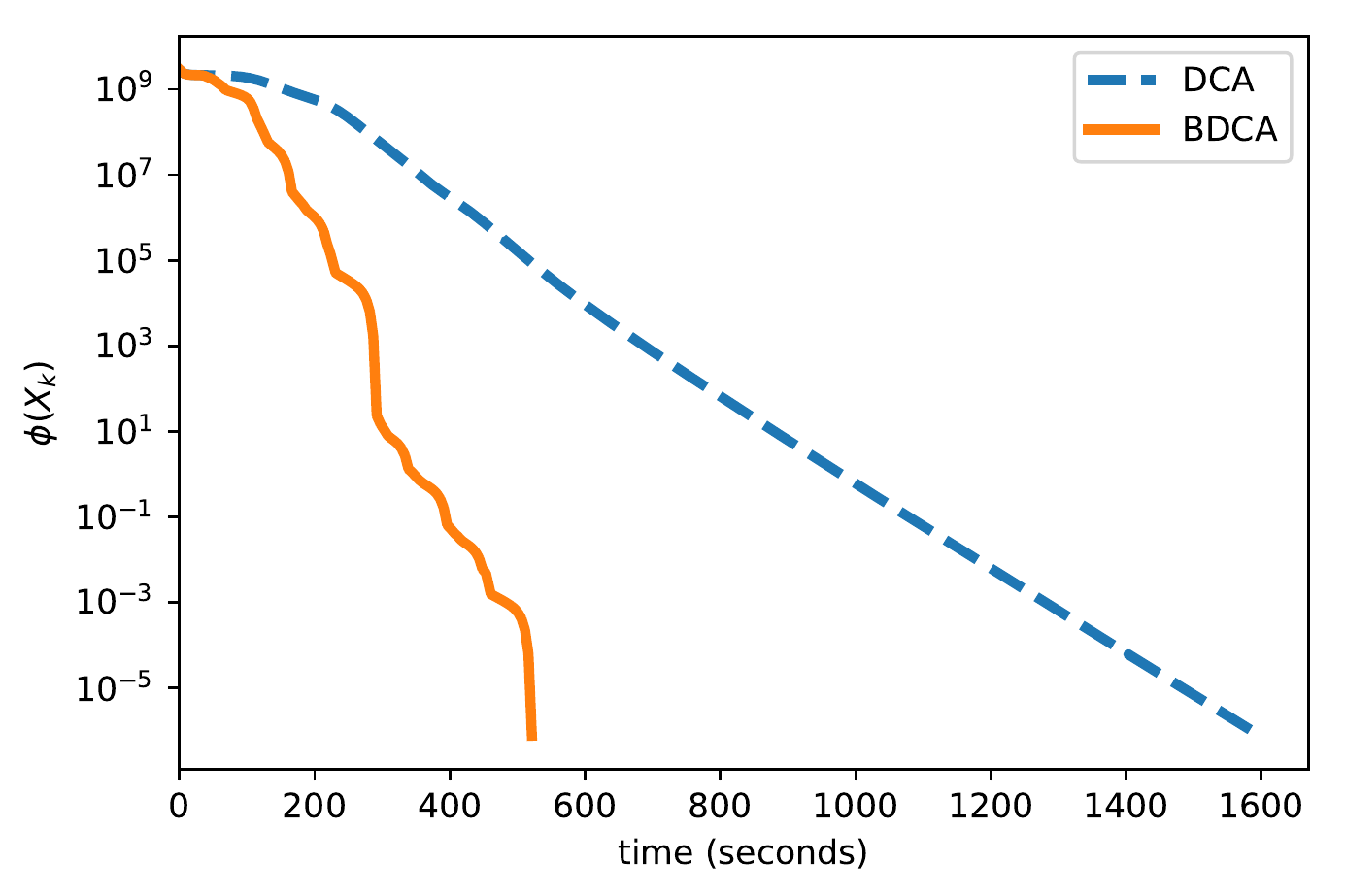}}\hfill
\subfigure[Case 2]{\includegraphics[height=0.32\textwidth]{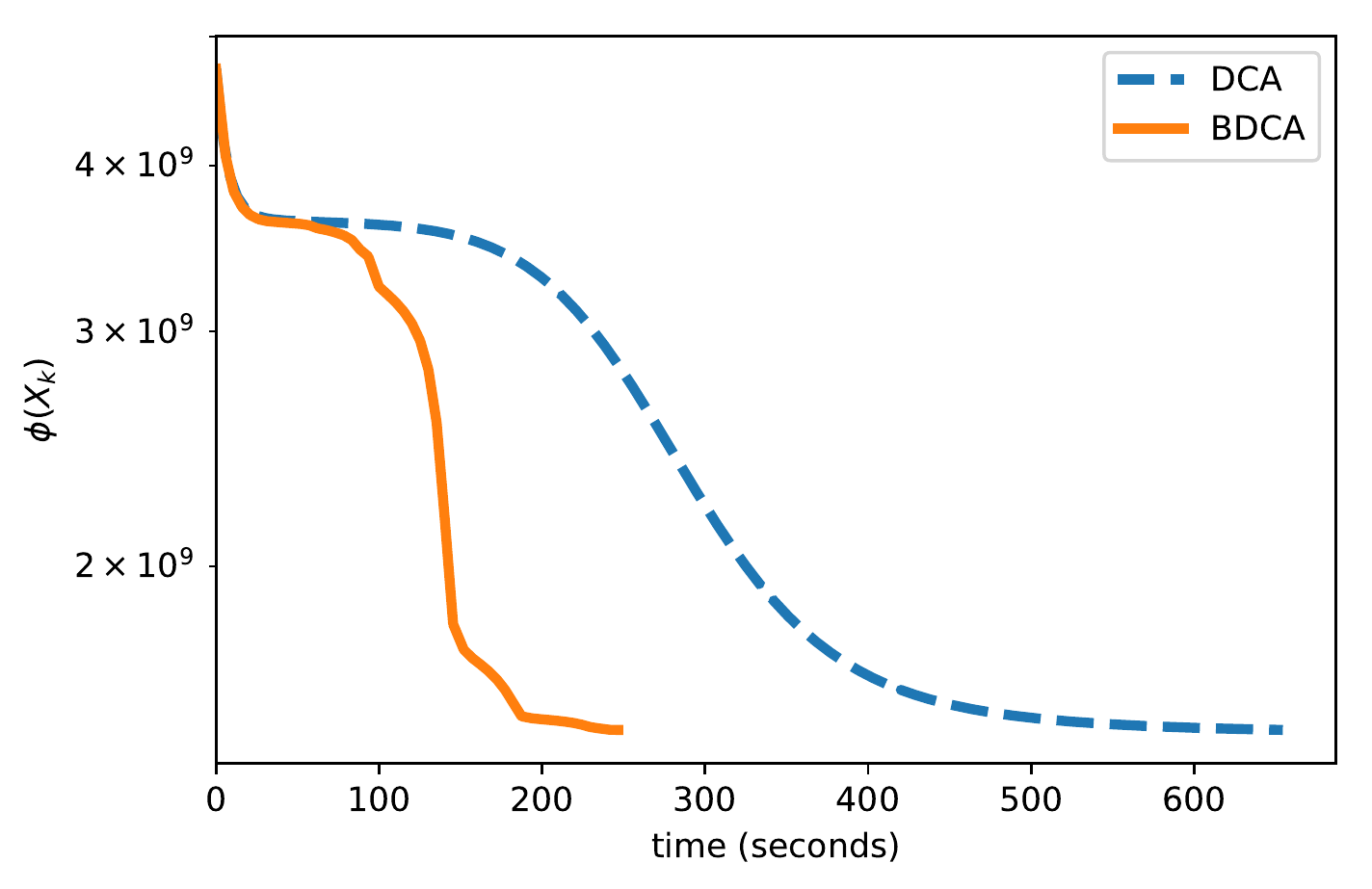}}
\caption{Value of the objective function of DCA and BDCA (using logarithmic scale) against CPU time for one particular random instance of each of the two test cases in \cref{exp:4} (with $p=3$ and $n= 10,\!000$).\label{fig:exp4_3}}
\end{figure}
\end{experiment}

\section{Concluding Remarks}\label{sec:remark}

We have developed a version of the Boosted DC Algorithm proposed in \cite{BDCA2018} for solving DC programming problems when the objective function is not differentiable. Our convergence results were obtained under some standard assumptions. The global convergence and convergent rate was established assuming the strong Kurdyka--\L{}ojasiewicz inequality. It remains as an open question whether the results still hold under the Kurdyka--\L{}ojasiewicz inequality, i.e., the corresponding inequality associated with the limiting subdifferential instead of the Clarke's one. This is a topic for future research.

We have applied our algorithm for solving two important problems in data science, namely, the Minimum Sum-of-Squares Clustering problem and the Multidimensional Scaling problem. Our numerical experiments indicate that BDCA outperforms DCA, being on average more than sixteen times faster in the first problem and nearly three times faster in the second problem, in both computational time and  number of iterations. In general, the advantage of BDCA against DCA will always depend on two key factors: the difficulty in solving the subproblems $(\mathcal{P}_k)$ and the number of backtracking steps needed at each iteration. A relatively small backtracking parameter $\beta\approx 0.1$ seems to work well in practice.

An important novelty of the proposed algorithm is the flexibility in the choice of the trial step size $\overline{\lambda}_k$ in the line search step of BDCA, which had to be constant in our previous work~\cite{BDCA2018}. A comparison of both strategies if shown in \cref{fig:compare_self-adaptive} using the same starting point as in \cref{fig:Spain}, where we can observe that each drop in the function value of the self-adaptive strategy was originated by a large increase of the step size. Although BDCA with constant choice was slower, it still needed three times less iterations than DCA, see \cref{fig:Spain}(a). The complete freedom in the choice of $\overline{\lambda}_k$ permits to use the information available from previous iterations, as done in \cref{sec:numa} with what we call the \emph{self-adaptive trial step size}. Roughly, this strategy allowed us to obtain a two times speed up of BDCA in all our numerical experiments, when compared with the constant strategy. There are many possibilities in the choice of the trial step size to investigate, which could further improve the performance of BDCA.

\begin{figure}[ht!]\centering
\includegraphics[width=0.95\textwidth]{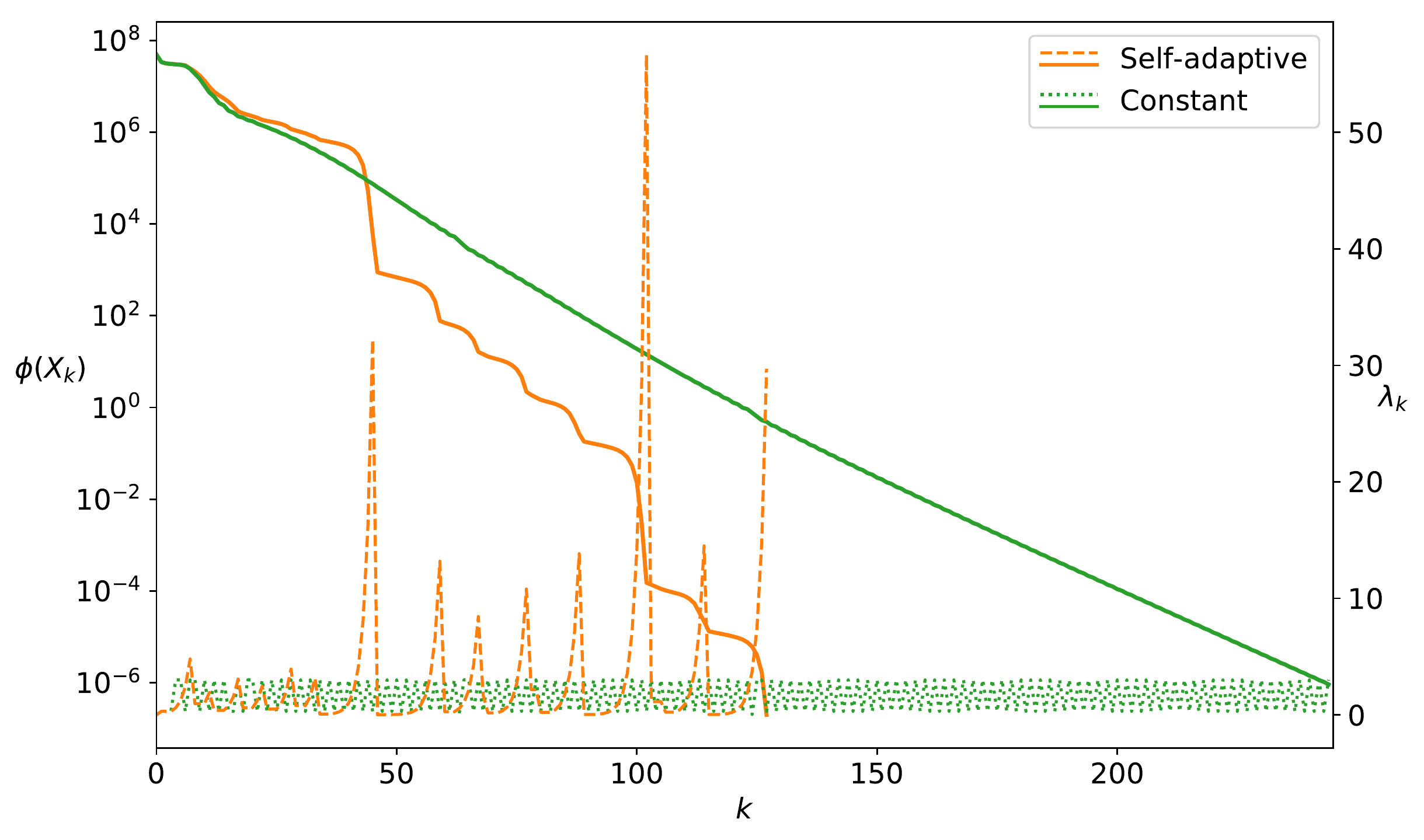}
\caption{Comparison of the self-adaptive and the constant (with $\overline{\lambda}_k=3$) choices for the trial step sizes of BDCA in Step~4, using the same starting point as in \cref{fig:Spain}. The plot includes two scales, a logarithmic one for the objective function values, and another one for the step sizes (which are represented with discontinuous lines).}\label{fig:compare_self-adaptive}
\end{figure}

Finally, we would like to mention that applications of BDCA to the
Bilevel Hierarchical Clustering problem~\cite{Nam2018} and the Multicast Network Design problem \cite{Geremew2018} can be also considered. However, due to the inclusion of a penalty and a smoothing parameter, the DC objective function associated with these problems changes at each iteration,  see \cite{Geremew2018,Nam2018} for details. Therefore, the applicability of BDCA should be justified in this setting. This serves as an interesting question for future research.

\section*{Acknowledgements}
The authors wish to thank Aris Daniilidis for his help with \cref{rem:Aris}.

The first author was supported by MINECO of Spain and ERDF of EU, as part of the Ram\'on y Cajal
program (RYC-2013-13327) and the Grants MTM2014-59179-C2-1-P and PGC2018-097960-B-C22.
The second author was supported by FWF (Austrian Science Fund), project I 2419-N32.

\bibliographystyle{siamplain}

\begin{thebibliography}{}

\bibitem{AnNam2017}
{\sc N.T. An and N.M. Nam}, \emph{Convergence analysis of a proximal point algorithm for minimizing differences
of functions}, Optimization, 66 (2017), pp.~129--147.


%
\bibitem{BDCA2018}
{\sc F.J. Arag{\'o}n Artacho, R. Fleming, and P.T. Vuong}, \emph{Accelerating the {DC}
  algorithm for smooth functions}, Math. Program., 169B (2018), pp.~95--118.
%


\bibitem{Attouch2010} {\sc H. Attouch, J. Bolte, P. Redont, and A. Soubeyran}, \emph{Proximal alternating minimization and projection
methods for nonconvex problems. An approach based on the Kurdyka--\L{}ojasiewicz inequality}, Math.
Oper. Res., 35 (2010), pp.~438--457.

\bibitem{attouch2009convergence}
{\sc H. Attouch and J. Bolte}, \emph{On the convergence of the proximal algorithm for nonsmooth functions involving analytic features}, Math. Program., 116 (2009), pp.~5--16.

\bibitem{BB2018} {\sc S. Banert and R. Bo\c{t}}, \emph{A general double-proximal gradient algorithm for d.c.
programming},  Math. Program., (2018), DOI \href{https://doi.org/10.1007/s10107-018-1292-2}{10.1007/s10107-018-1292-2}.

\bibitem{Bock1998}  {\sc H.H. Bock}, \emph{Clustering and neural networks}, in Advances in Data Science and
  Classification, Springer, Berlin, 1998, pp.~265--277.
%
\bibitem{bolte2007lojasiewicz}
{\sc J. Bolte, A. Daniilidis, and A. Lewis}, \emph{The \L{}ojasiewicz inequality for nonsmooth subanalytic functions with applications to subgradient dynamical systems},
 SIAM J. Optimiz., 17 (2007), pp.~1205--1223.
%

\bibitem{bolteArisLewis2007}
{\sc J. Bolte, A. Daniilidis, A. Lewis, and M. Shiota}, \emph{Clarke subgradients of stratifiable functions}, SIAM J. Optim., 18 (2007), pp.~556--572.

\bibitem{BDLM10}
{\sc J. Bolte, A. Daniilidis, O. Ley, and L. Mazet}, \emph{Characterizations of Lojasiewicz inequalities: subgradient flows, talweg, convexity}, Trans. Amer. Math. Soc., 362 (2010), pp.~3319--3363.

\bibitem{Bolte2013}
{\sc J. Bolte, S. Sabach, and M. Teboulle}, \emph{Proximal alternating linearized minimization for nonconvex and
  nonsmooth problems}, Math. Program., 146 (2013), pp.~459--494.

\bibitem{Clarke}
{\sc F.H. Clarke}, \emph{Optimization and Nonsmooth Analysis}, Second edition, Classics Appl. Math. 5, SIAM, Philadelphia, 1990.

\bibitem{CYY2018}
{\sc T.H. Cuong, N.D. Yen, and Y.C. Yao}, \emph{Qualitative Properties of the Minimum
Sum-of-Squares Clustering Problem}, arXiv: \href{https://arxiv.org/abs/1810.02057}{1810.02057}
%
\bibitem{fukushima_generalized_1981}
{\sc M. Fukushima and H. Mine}, \emph{A generalized proximal point algorithm for certain non-convex
  minimization problems}, Int. J. Syst. Sci., 12 (1981), pp.~989--1000.


\bibitem{Geremew2018}
{\sc W. Geremew, N.M. Nam, A. Semenov, V. Boginski, and E. Pasiliao}, \emph{A DC programming approach for solving multicast
network design problems via the Nesterov smoothing
technique}, J. Glob. Optim., 72 (2018), pp.~705--729.

\bibitem{Kurdyka}
{\sc K. Kurdyka}, \emph{On gradients of functions definable in o-minimal structures}, Annales de l'Institut Fourier
(Grenoble), 48 (1998), pp.~769--783.

\bibitem{JOTA2018}
{\sc H.A. Le~Thi, V.N. Huynh, and T. Pham~Dinh}, \emph{Convergence analysis of Difference-of-Convex Algorithm with subanalytic data}, J. Optim. Theory Appl., 179 (2018), pp.~103--126.

\bibitem{Tao01}
{\sc H.A. Le~Thi, and T. Pham~Dinh} \emph{D.C. programing approach to the multidimensional scaling problem}, in From Local to Global Optimization, P. Pardalos and P. Varbrand, eds, Kluwer, Dodrecht, 2001, pp.~231--276.

\bibitem{tao2005dc}
{\sc H.A. Le~Thi and T. Pham~Dinh}, \emph{The DC (difference of convex functions) programming and DCA revisited with DC models of real world nonconvex optimization problems}, Ann. Oper. Res., 133 (2005), pp.~23--46.

\bibitem{An2018}
{\sc H.A. Le~Thi and T. Pham Dinh}, \emph{DC Programming and DCA: Thirty Years of
Developments}, Math. Program., 169 (2018), pp.~5--68.

\bibitem{An2014}
{\sc H.A. Le~Thi and T. Pham Dinh}, \emph{ Recent advances in DC programming and DCA},
Nguyen N-T, Le Thi HA, eds. Trans. Comput. Collective Intelligence Lecture
Notes in Computer Science, Vol. 8342 (Springer, Berlin),  pp. 1--37, 2014.

\bibitem{an_numerical_1996}
{\sc H.A. Le~Thi, T. Pham Dinh, and L.D. Muu}, \emph{Numerical solution for optimization over the efficient set by D.C. optimization algorithms}, Oper. Res. Lett., 19 (1996), pp.~117--128.

\bibitem{lojasiewicz1965ensembles}
{\sc S. \L{}ojasiewicz}, {\em Ensembles semi-analytiques}, Institut des Hautes Etudes
Scientifiques, Bures-sur-Yvette (Seine-et-Oise), France, 1965.
%
\bibitem{mine_minimization_1981}
{\sc H. Mine and M. Fukushima}, \emph{A minimization method for the sum of a convex function and a
  continuously differentiable function}, J. Optim. Theory Appl., 33 (1981), pp.~9--23.



\bibitem{moudafi_proximalDC_2006}
{\sc A. Moudafi and P. Mainge}, \emph{On the convergence of an approximate proximal method for DC
  functions}, J. Comput. Math., 24 (2006), pp.~475--480.

\bibitem{Nam2018}
{\sc N.M. Nam, W. Geremew, R. Reynolds, and T. Tran}, \emph{Nesterov's smoothing technique and minimizing
differences of convex functions for hierarchical clustering},  Optim. Lett., 12 (2018), pp.~455--473.

\bibitem{Noll2014}
{\sc D. Noll}, \emph{Convergence of non-smooth descent methods using the Kurdyka--\L{}ojasiewicz  inequality},
 J. Optim. Theory Appl., 160 (2014), pp.~553--572.

\bibitem{OB15}
{\sc B. Ordin and A.M. Bagirov}, \emph{A heuristic algorithm for solving the minimum sum-of-squares clustering problems}, J. Glob. Optim., 61 (2015), pp.~341--361.

\bibitem{tao1998dc}
{\sc T. Pham~Dinh and H.A. Le~Thi}, \emph{A DC optimization algorithm for solving the trust-region subproblem}, SIAM J. Optim., 8 (1998), pp.~476--505.


\bibitem{TR}
{\sc R.T. Rockafellar}, \emph{Convex Analysis}, Princeton University Press, 1972.

\bibitem{RW}
{\sc R.T. Rockafellar and R.J.-B. Wets}, \emph{Variational Analysis}, Grundlehren Math. Wiss. 317,
Springer, New York, 1998.

\bibitem{TT97}
{\sc P.D. Tao and H.A. Le Thi}, \emph{Convex analysis approach to DC programming: theory, algorithms and applications}, Acta Mathematica Vietnamica,
22 (1997), pp.~289--355.

\bibitem{Toland79}
{\sc J.F. Toland}, \emph{On subdifferential calculus and duality in non-convex optimization}, Bull. Soc. Math. Fr. M\'em. 60  (1979) (Proc. Colloq., Pau 1977), pp.~177--183.

\bibitem{XuXue2017}
{\sc H.M. Xu, H. Xue, X.H. Chen, Y.Y. Wang}, \emph{Solving Indefinite Kernel Support Vector Machine with
Difference of Convex Functions Programming}, Proceedings of the Thirty-First AAAI Conference on Artificial Intelligence (AAAI-17).
\end{thebibliography}

\end{document}